\documentclass[AMA,STIX1COL,12pt]{WileyNJD_v2}
\usepackage{amsmath,amsfonts,mathtools,lipsum}
\usepackage{graphicx}
\usepackage{textcomp}
\usepackage{url}
\usepackage{cuted}
\usepackage{float}
\usepackage[numbers]{natbib}
\usepackage{academicons}
\usepackage{IEEEtrantools}
\usepackage[subnum]{cases}

\hypersetup{pdfauthor={Name}}
\usepackage{hyperref}

\newcommand{\bx}{{\bf x}}
\newcommand{\by}{{\bf y}}
\newcommand{\bz}{{\bf z}}
\newcommand{\bu}{{\bf u}}
\newcommand{\bv}{{\bf v}}
\newcommand{\bw}{{\bf w}}

\newcommand{\boldell}{\boldsymbol\ell}
\newcommand{\boldtheta}{\boldsymbol\theta}
\newcommand{\boldeta}{\boldsymbol\eta}
\newcommand{\boldxi}{\boldsymbol\xi}

\newcommand*{\citen}[1]{%
  \begingroup
    \romannumeral-`\x 
    \setcitestyle{numbers,square}%
    \cite{#1}%
  \endgroup   
}

\newcommand*{\citem}[1]{%
  \begingroup
    \romannumeral-`\x 
    \setcitestyle{numbers}%
    \cite{#1}%
  \endgroup   
}

\articletype{Sparsity-Inducing Methods}%

\begin{document}

\title{Optimized Data Rate Allocation for Dynamic Sensor Fusion over Resource Constrained Communication Networks}

\author[1]{Hyunho Jung*}

\author[1]{Ali Reza Pedram}

\author[2]{Travis C. Cuvelier}

\author[3]{Takashi Tanaka}

\authormark{JUNG \textsc{et al}}

\address[1]{\orgdiv{Walker Department of Mechanical Engineering}, \orgname{University of Texas at Austin}, \orgaddress{\state{78712 Texas}, \country{U.S.A.}}}


\address[2]{\orgdiv{Department of Electrical and Computer Engineering}, \orgname{University of Texas at Austin}, \orgaddress{\state{78712 Texas}, \country{U.S.A.}}}

\address[3]{\orgdiv{Department of Aerospace Engineering and Engineering Mechanics}, \orgname{University of Texas at Austin}, \orgaddress{\state{78712 Texas}, \country{U.S.A.}}}

\corres{*Hyunho Jung, Department of Mechanical Engineering, University of Texas at Austin, 78712 Texas, U.S.A. \email{jung.hyunho@utexas.edu}}

\fundingInfo{Defense Advanced Research Projects Agency, Grant Number: D19AP00078; National Science Foundation, Award Number: 1944318}


\abstract[Abstract]{This paper presents a new method to solve a dynamic sensor fusion problem. We consider a large number of remote sensors which measure a common Gauss-Markov process. Each sensor encodes and transmits its measurement to a data fusion center through a resource restricted communication network.  The communication cost incurred by a given sensor is quantified as the expected bitrate from the sensor to the fusion center. We propose an approach that attempts to minimize a weighted sum of these communication costs subject to a constraint on the state estimation error at the fusion center. We formulate the problem as a difference-of-convex program and apply the convex-concave procedure (CCP) to obtain a heuristic solution. 
We consider a 1D heat transfer model and a model for 2D target tracking by a drone swarm for numerical studies. Through these simulations, we observe that our proposed approach has a tendency to assign zero data rate to unnecessary sensors indicating that our approach is sparsity-promoting, and an effective sensor selection heuristic.}


\jnlcitation{\cname{
\author{Jung H}, 
\author{Pedram AR}, 
\author{Cuvelier TC},
\author{Tanaka T.}}, 
\ctitle{Optimized Data Rate Allocation for Dynamic Sensor Fusion over Resource Constrained Communication Networks}, \cjournal{Int J Robust Nonlinear Control.}, \cvol{2022;1-27}, \url{doi:10.1002/rnc.6076}}

\maketitle


\section{Introduction}\label{sec:introduction}
In this work, we develop techniques to efficiently allocate communication resources in dynamic sensor fusion over a resource-constrained network. Our principal motivation is networked autonomous systems where the communication between sensors and controllers is wireless. Networked autonomous systems play a major role in the heavy industry, transportation, and  defense sectors. In such settings, wired communication may be impossible due to physical constraints (such as mobility). Wired communication may also be impractical due to the difficulty of routing and maintaining cables. This latter concern is not without consequence; a report from the US Navy  indicates that over 1,000 missions are aborted per year due to wiring faults \citen{navy}. Meanwhile, as control systems become more automated and perform more complex tasks, the number of sensor platforms tends to increase. For example, an Airbus A380-1000 model has about 10,000 sensors on its wings alone \citen{a380}. In wireless systems, communication resources are inherently constrained. Understanding the impact of control systems on wireless networks, and developing strategies to manage wireless resources efficiently subject to the particular requirements of control, will help enable the development of future autonomous systems.  

We first give a high-level overview of our contributions before reviewing the literature. 
\subsection{Our contribution}
We consider a setup where several remote sensor platforms observe a common Gauss-Markov process and forward these observations to a fusion center over reliable point-to-point binary communication channels. The use of these \textit{uplink} channels incurs a communication cost, quantified as the expected bitrate of a uniquely decodable source code. Given the observations it receives, the fusion center forms an estimate of the state. Using a \textit{downlink} channel, which we assume incurs no cost, the fusion center feeds this estimate back to the sensor platforms. This formulation is applicable to a variety of physical layer resource allocation problems as detailed in (Sections \ref{ssec:dfrc} and \ref{ssec:phy}). Our key contributions are presently summarized: 
\begin{itemize}
    \item We formulate a sensor rate allocation problem. Using entropy-coded dithered predictive quantization (cf. \citen{zamir1992}, \citen{tanaka2016}) we show that the expected bitrate required to encode a quantized linear measurement with a specified reconstruction error can be approximated by a particular information measure.
    \item We propose to minimize a weighted average of these bitrates, over both sensors and a time horizon, subject to a constraint on the error covariance of an linear minimum mean square error estimator at the fusion center. This results in a non-convex optimization.
    \item We convert the non-convex optimization to difference-of-convex program \citen{lipp2016} and optimize via the heuristic convex-concave procedure (CCP). While there is no guarantee that the CCP procedure will produce the global minimum, it is guaranteed to produce a feasible local minimum \citen{lipp2016}.
    \item We perform numerical experiments where we apply our approach to a time-invariant heat transfer system and a time-varying drone tracking system. We observed that as the estimator performance constraint is loosened, our algorithm tends to allocate zero data rate to more unneeded sensors. Our proposed approach thus tends to be sparsity-promoting, and may be used as a sensor selection heuristic. 
    \item We provide analytical insight into the sparsity-promoting aspect of our approach via considering limiting examples in the scalar case.  
\end{itemize}

\subsection{Literature review}
A portion of this work appears in \citen{hjConference}. In this manuscript, we motivate our problem formulation more concretely with an application to physical layer resource allocation in wireless communications (cf. Section \ref{ssec:phy}). We also include new, more elucidating simulation results (cf. Section \ref{sec:simulations}). Finally, we conclude with an additional discussion of the sparsity-promoting property of our approach (cf. Section \ref{sec:discussion}). 

The relevant prior work straddles several research areas, including optimal sensor selection and control with communication constraints,  network information theory (in particular, multiterminal rate-distortion theory), and resource allocation in wireless communications. We highlight some relevant contributions from this areas and contrast them with our present work in the remainder of this section. 

\subsubsection{Optimal sensor selection and control with communication constraints }
Optimal sensor selection has been a well-studied problem for several decades. Generally speaking, such approaches aim to choose an optimal subset of sensors from some set. Information measures are often incorporated, via either the objective functions or the constraints, in the optimization over subsets. Early examples include heuristic subset optimizations that attempted to maximize the trace or determinant of the Fisher information matrix (FIM). For example, \citen{kammer1990} proposed a sensor selection method based on ranking sensors at different locations in terms of a surrogate for their contribution to the determinant of the FIM. By an iterative technique, the method removes insignificant sensors, resulting in a selected subset of sensors that tend to maximize the trace and determinant of the FIM. In \citen{yao1992} sensor selection was performed via a genetic algorithm, with the determinant of the FIM as the figure of merit. 

More relevant recent results include those that consider Bayesian estimation/tracking. In \citen{wang2004}, an entropy-based heuristic for sensor selection was introduced and applied to target localization. A semidefinite programming (SDP) relaxation of a sensor selection problem aiming to minimize the determinant of the estimation error covariance matrix was proposed in \citen{joshi2009}. In a target tracking problem, \citen{hoffmann2010} used an estimate of the mutual information between a sensor's observation and the target state  (computed via particle filter) to evaluate a notion of information gain. In \citen{qingli2020}, an optimization method was proposed which maximizes Bayesian Fisher Information and mutual information while minimizing the number of selected sensors. The system model introduced in \citen{mo2011} resembles the one in our present work. In \citen{mo2011}, a fusion center tracks a Gauss-Markov process with incoming observations from multiple remote sensors. A finite horizon optimization problem is proposed to identify a subset of sensors to be transmit their measurements at each time step. The problem is solved heuristically by a re-weighted $\ell_1$ relaxation.  Sensor selection in aircraft engine health monitoring was studied in \citen{liu2016}. In particular, the selection was performed via evaluating entropy.  

In this work we do not consider sensor selection, but rather \textit{sensor rate allocation}. Rather than selecting a subset of sensors, we formulate an optimization problem to minimize a weighted average of data rates from the sensor platforms to the controller. Our present formulation offers additional flexibility; by choosing different weights, one can apply our present formulation to a variety of physical layer resource allocation problems (cf. \ref{ssec:phy}). Notably, we also consider a system model where the sensors have strictly causal feedback access to the fusion center's best estimate. In Section \ref{ssec:dfrc}, we argue that this is a reasonable assumption in several interesting and practical regimes. 

Some of the prior art on control with communication constraints is also relevant to this work. The use of prefix-free coding and entropy-coded dithered quantization (ECDQ) for a control subject to data rate constraints was motivated by \citen{silvaFirst}. The work in \citen{silvaFirst} was extended to the case of MIMO plants by \citen{tanaka2016}. In \citen{stefan2019sparse},  information-regularized control of a distributed system was studied with the objective of reducing inter-subsystem communication while maintaining adequate control cost. The notion of communication cost in \citen{stefan2019sparse} resembles the one of this present work. It was demonstrated that the optimal controller that jointly minimizes the control cost and the required data rate for inter-subsystem communications has a sparse structure. 
While the communication cost in our present work is similar to the one in \citen{stefan2019sparse},  we consider a significantly different scenario---namely a distributed sensing paradigm where multiple independent sensor platforms independently convey their measurements to a fusion center. Our objective is to minimize bitrate while maintaining a constraint on the estimator mean squared error (MSE). 
 
\subsubsection{Multiterminal rate-distortion}\label{ssec:ceoback}
In this work, we address a problem that closely resembles the quadratic Gaussian CEO (Chief Executive/Estimation Officer) problem \citen{ceo}. In the CEO problem, the decoder, corresponding to the ``executive'' (analogous to the fusion center in this work), receives messages from a set of encoders or ``agents''. Each agent observes a measurement of some random state variable that the executive would like to reconstruct. The tradeoff of interest is the sum rate of communication (usually measured in terms of fixed-length coding) versus the distortion in the CEO's estimate. Generally, it is assumed that the agents' measurements are conditionally independent, given the state, and that agents are not allowed to pool their data beforehand. In the ``quadratic Gaussian'' version of this problem (cf. \citen{gaussceo}), agents observe a Gaussian process subject to independent additive Gaussian noise. The distortion metric is (block average) mean squared error. With an eye to distributed tracking, \citen{kostina} analyzed a causal version of the general CEO problem. A rate-distortion formulation is proposed where Massey's directed information (cf. \citen{masseyDI}) is minimized subject to a constraint on the CEO's estimator distortion. This tradeoff is shown to be computable via a convex optimization in the Gaussian case. The rate loss due to the lack of communication between the sensors is also analyzed. 

In this work, we deviate from the formulation of the CEO problem in that we minimize a weighted average of the data rates from sensor platforms to the fusion center, rather than just the sum rate. This is advantageous, for example, in an application to physical layer resource allocation (cf. Section \ref{ssec:phy}). Most coding schemes for the CEO problem use fixed-length coding. We consider zero-delay uniquely decodable variable-length source coding, which may be advantageous in practice. Finally, we consider a setting where the sensors have feedback access to the decoder's estimate.  Unfortunately, our lower bounds on codeword length are not completely general (they are with respect to a particular quantization scheme), and we have as-yet been unable to convexify the relevant optimization problem we derive. 

It is worth mentioning that the CEO problem and the problem considered in this manuscript are distinct from the standard vector Gaussian rate-disortion problem (which admits the well known reverse waterfilling solution). Both the CEO problem and the problem considered in this manuscript are inherently sequential; they consider tracking a dynamical system rather than and memoryless source. Even with a time-horizon of $1$, the CEO problem and the scenario under investigation are fundamentally distributed source coding problems; in standard vector Gaussian rate-distortion, the compressor has access to the entire source vector. In contrast, in the present work we assume that a set of independent compressors each has access to a scalar, linear, and potentially noisy measurement of the source vector.

\subsubsection{Resource allocation for control, estimation, and detection over wireless networks} Ultra-low latency reliable wireless communication (ULLRC) has been proposed as an enabler of next-generation networked autonomous systems---in particular for control systems using wireless communication \citen{yilmaz} \citen{durisi2016toward}. In the 5G cellular context, ULLRC is characterized by very short end-to-end latencies (on the order of a millisecond) and extremely small error probabilities (on the order of $10^{-5}$) \citen{durisi2016toward}. In \citen{ullrc_prediction} a prediction-communication co-design approach was proposed with a network topology similar to the one considered in our present manuscript. A non-convex optimization was introduced to minimize the bandwidth required to decode a fixed-length message subject to constraints on block error probability and the probabilities of exceeding the latency or prediction error constraints.  This approach was shown to effectively improve the tradeoff between reliability and performance (measured in terms of the aforementioned error rate metrics). In this work we consider a different notion of performance, namely estimator MSE. Our scheme is also able to dynamically allocate wireless resources by means of effective source coding; we optimize the actual number of bits conveyed from the sensors to the fusion center. In contrast to \citen{ullrc_prediction}, we do not include finite-blocklength channel coding in our analysis. In \citen{rebal}, a general feedback control system using ULLRC is abstracted analogously to the CEO problem. The executive is assumed to operate on a fixed schedule and is required to produce its actions by a hard real-time deadline. Different agents make asynchronous linear/Gaussian measurements of a Gauss-Markov dynamical system, and transmit their measurements to the executive over a shared communication channel. The executive's estimator performance is measured in terms of MSE. Polling an observer incurs a communication cost  (assumed known a priori), or \textit{airtime}, and a branch and bound approach is applied to schedule the optimal sequence of observers that can be polled before the deadline. Our present work could be seen as optimizing the total (expected) airtime required to achieve some estimator performance (cf. Sec. \ref{ssec:phy}). 

The recent survey \citen{procSmartManufacturing} on control systems incorporating wireless communication contains many relevant references, several practical case studies, and a detailed overview of a framework for control/communication co-design. It is argued that wireless systems for control applications should be dependable (measured in terms of control performance and stability), adaptive (reconfigurable), and efficient (with respect to the use of resources like time, bandwidth, and power). A control-guided communication approach was proposed where messages are transmitted according to a self-triggered protocol. The approach was shown to be effective through experiments in a cyber-physical testbed. In this work, we consider synchronous communication, which may be a necessity for control over congested networks. Our sensor rate allocation scheme is also inherently dynamic; in simulation, we demonstrate that different subsets of sensors are assigned nonzero bitrates over time. Furthermore, our choice to minimize bitrate allows us to draw a direct link between the utilization of physical layer resources and estimator performance, improving efficiency (at least on average). 

Some prior work on sensor fusion for decentralized detection is particularly relevant to our present investigation. In \citen{tcomm_detection_wsn}, a detection problem is considered where samples of a continuous-time signal are quantized at different rates over time. Given a sampling strategy and a constraint on total bitrate, different fixed-length quantizers are designed for each temporal sample using a criterion based on minimizing a notion of information loss. The approach in \citen{tcomm_detection_wsn} quantifies information loss in terms of general $f$-divergences. In \citen{TAES_detection_wsn} several sensor platforms observe noisy scalar measurements of some deterministic unknown parameter. These measurements are conveyed to a fusion center whose goal is to test a hypothesis based on the deviation of the parameter from some known normal condition. A class of ``smart'' sensors transmits unquantized measurements over error free channels, meanwhile a class of ``dumb'' sensors transmits (fixed-length) quantized measurements over binary symmetric channels. Both a Generalized Likelihood Ratio Test (GLRT) and a (computationally simpler) Rao test for the aformentioned model are derived and analyzed, and an optimal quantizer design is pursued through nonconvex optimization. In \citen{TSPN_detection_wsn} a similar setup is considered with the measurement model extended to include  (possibly) nonlinear, noisy, vector measurements of the deterministic unknown parameter. In \citen{TSPN_detection_wsn}, each sensor can apply a linear precoder to its measurement before performing a fixed-length quantization. The quantized measurements are conveyed to the fusion center over a noisy binary symmetric channel. Again, both a GLRT and a (computationally simpler) Rao test are derived and analyzed. In particular, an asymptotic performance characterization lead to an optimal design for a one-bit quantizer (given some choice of precoder). The one-bit results are extended to multi-bit quantizers, and numerical study was performed to demonstrate that the Rao test achieves comparable performance to the GLRT.

In contrast to \citen{tcomm_detection_wsn}, \citen{TAES_detection_wsn}, and \citen{TSPN_detection_wsn} the focus of our present manuscript is on estimation and tracking of a dynamical system as opposed to detection. Furthermore, we consider variable-length source coding, in contrast to fixed-length quantization.


\subsection{Organization and notation}
This paper is organized as follows. Section \ref{sec:formulation} reviews ECDQ, defines the notion of communication cost, and connects this cost to the mutual information (MI) between two Gaussian random variables that arise in Kalman filtering. The sensor rate allocation (SRA) problem is proposed. We conclude the section with an application example motivated by physical layer resource allocation in a remote sensing scenario. In Section \ref{sec:doc}, we show that the communication cost can be written in terms of error covariance matrices from Kalman filtering. In these terms, we convert the SRA optimization to a difference-of-convex program. Section \ref{sec:ccp} defines  
an iterative heuristic algorithm based on the convex-concave procedure (CCP) to attack the resulting non-convex optimization. Our numerical studies are introduced in Section \ref{sec:simulations}. We provide insight into the sparsity-promoting nature of our approach in Section \ref{sec:discussion}. We conclude in Section \ref{sec:conclusion}. 

Lower case boldface symbols such as $\bx$ denote random variables and use the notation $\bx_{1:t}=(\bx_1, ... , \bx_t)$ for subsets of random variables from some discrete-time random process. We use standard information-theoretic notation from \citen{cover2012elements}: the entropy of a discrete random variable $\bx$ is denoted by $H(\bx)$, while the differential entropy of a continuous random variable $\bx$ is denoted by $h(\bx)$. The mutual information between $\bx$ and $\by$ is denoted by $I(\bx;\by)$, and the relative entropy is denoted by $D(\cdot \| \cdot)$. $\mathbb{S}^n$ represents the set of symmetric matrices of size $n\times n$. $X\in\mathbb{S}_+^n$ or $X\succeq 0$  indicates that $X$ is a positive semidefinite matrix, and  $X\in\mathbb{S}_{++}^n$ or $X\succ 0$  represents that $X$ is a positive definite matrix. We denote the set of natural numbers $\{1, 2, \dots, M\}$ as $1:M$, and the set of positive natural numbers as $\mathbb{N}^+$.

\section{Problem Formulation}\label{sec:formulation}
We consider discrete-time estimation of a Gauss-Markov source and assume the source is an $n$-dimensional random process defined via 
\begin{equation}
\label{eq:source}
\bx_{t+1}=A\bx_t+F\bw_t, \;\; \bw_t \stackrel{i.i.d.}{\sim} \mathcal{N}(0,I), \;\; t=1, 2, ... ,T
\end{equation}
with initial condition $\bx_1\sim \mathcal{N}(0, P_{1|1})$, where $P_{1|1}\in \mathbb{S}_+^n$ and $A\in\mathbb{R}^{n\times n}$.
We consider remote estimation  over a star-like sensor network, as shown in  Fig~\ref{fig:network}. Each of the  $M$ sensors makes a 
noisy, scalar, linear measurement of the state vector. At time $t$, sensor $i$ observes a linear measurement of the plant with additive Gaussian noise. Denote the sensor's measurement matrix by $C_{i}\in\mathbb{C}^{1\times n}$ and let $U_{i} \in \mathbb{S}^n_+$ and denote the measurement noise covariance. Let the measurement noise at the platform be given by $\bu_{i,t}\sim \mathcal{N}(0, U_{i})$ i.i.d over time. The $i^{\mathrm{th}}$ sensor platform observes $\by_{i,t}=C_i\bx+ \bu_{i,t}$. Concatenating the measurements from each sensor into $\mathbf{y}_{t}\in\mathbb{C}^{M}$, we have 
\begin{subequations}\label{eq:measurements}
\begin{align}
    \by_t=C\bx_t + \bu_{t},
\end{align} where 
\begin{equation}
\by_t=\begin{bmatrix}\by_{1,t} \\ \vdots \\ \by_{M,t}\end{bmatrix},  \;\;
C=\begin{bmatrix}C_1 \\ \vdots \\ C_M \end{bmatrix},\;\;
\bu_t=\begin{bmatrix}\bu_{1,t} \\ \vdots \\ \bu_{M,t}\end{bmatrix}.
\end{equation}
\end{subequations} We assume that the noises at different sensors are independent, e.g. we assume that for all $t$
\begin{IEEEeqnarray}{rCl}\label{eq:umatdef}
   U &=& \mathbb{E}(\bu_t\bu^{\mathrm{T}}_t)\\ &=& \begin{bmatrix}
    U_1 & 0 & \dots & 0 \\
    0 & U_2 & \dots & 0 \\
    \vdots & \vdots & \ddots & \vdots \\
    0 & 0 & \dots & U_M 
\end{bmatrix}.\nonumber
\end{IEEEeqnarray}
We assume that sensors must convey these measurements to the fusion center over a resource-constrained communication network, as detailed in the sequel. Throughout the paper, we assume that the pair (A,C) is observable. 

\subsection{Data fusion over resource-constrained network}\label{ssec:dfrc}
We assume the platforms operate synchronously in discrete-time, analogous to a CAN-bus-like system. At time $t$, we assume that sensor $i$ encodes $\by_{i,t}$ into a binary codeword, say $\bz_{i,t}$, from a uniquely decodable variable-length code. The length of the codeword $\bz_{i,t}$, in bits, is denoted  $\ell_{i,t}$. The expected length of $\ell_{i,t}$ measures the bitrate of uplink communication from the $i^{\mathrm{th}}$ sensor platform to the fusion center, and provides a notion of communication cost. In the CAN protocol, a packet frame consists of a header and a tailer in addition to the data field. For the sake of simplicity, we neglect the contributions from these components to the network's bitrates. We assume that the uplink communication is \textit{reliable} in the sense that the packet $\bz_{i,t}$ is received without error before time $t+1$. 

Having received  packets from each of the $M$ platforms, the fusion center decodes packets and computes the linear minimum mean-square error (LMMSE) estimate $\hat{\bx}_{t|t}$. The fusion center also calculates a step-ahead prediction $\hat{\bx}_{t+1|t}$ based on least mean-square-error estimate $\hat{\bx}_{t|t}$ and the source process \eqref{eq:source}. Once the prediction $\hat{\bx}_{t+1|t}$ is calculated, the fusion center transmits the prediction back to all the sensors as shown in Fig.~\ref{fig:network}. Access to $\hat{\bx}_{t+1|t}$ at the sensor platforms allows them to apply  \emph{predictive quantization} at time step $t+1$. Quantization and encoding the \emph{innovation} $\by_{i,t+1}-C_{i}\hat{\bx}_{t+1|t}$, rather than $\by_{i,t+1}$ directly, helps  to reduce the expected codeword length (cf. \citen{yuksel2013} and \citen{tanaka2016}).

We neglect the communication cost associated with the downlink communication and assume that the sensors receive $\hat{\mathbf{x}}_{t+1|t}$ exactly. These assumptions are appropriate in a variety of circumstances. In one such scenario, the remote sensor platforms could be battery-powered meanwhile the fusion center could have direct access to power lines. If the communication between the sensors and the fusion center is wireless, the sensor platforms could have strictly constrained transmit powers meanwhile the fusion center could have no such constraint. In such a setting, the fusion center could encode an extremely finely quantized version of $\hat{\mathbf{x}}_{t+1|t}$ and use an extremely high power to transmit this message reliably to all of the $M$ sensor platforms. In a scenario with time-division duplexing, the uplink and downlink time-share the same frequency channel. As we discuss in Section \ref{ssec:phy}, a reasonable notion of communication ``cost'' in terms of physical layer resources used is the amount of time the frequency channel is occupied by uplink and downlink transmissions. We discuss this in more detail in Section \ref{ssec:phy}, however, at this stage it suffices to recognize that if the downlink transmit power is unconstrained, an arbitrary number of bits can be conveyed from the fusion center to the sensor platforms in a minimal amount of time.

It is also reasonable to neglect the downlink communication cost in a wireless communication scenario where each remote sensor platform can easily ``overhear'' (e.g. decode reliably) the transmissions from other platforms ostensibly directed to the fusion center. Formally, this is possible when, for each remote platform, the multicast capacity of the channel from the given platform to the others exceeds the channel capacity from the given platform to the fusion center. Essentially, this means that the channel capacity between any two platforms exceeds the capacity between any one platform and the fusion center. This occurs, for example, in a line-of-sight communication setting where the distance between the individual sensor platforms is small compared with the distance between the platforms and the fusion center. In this case, each user can form the estimate $\hat{\mathbf{x}}_{t+1|t}$ simply based on the transmissions it overhears on the uplink---no downlink communication is actually required.

\begin{figure*}[t]
    \centering
    \includegraphics[width=1.0\columnwidth]{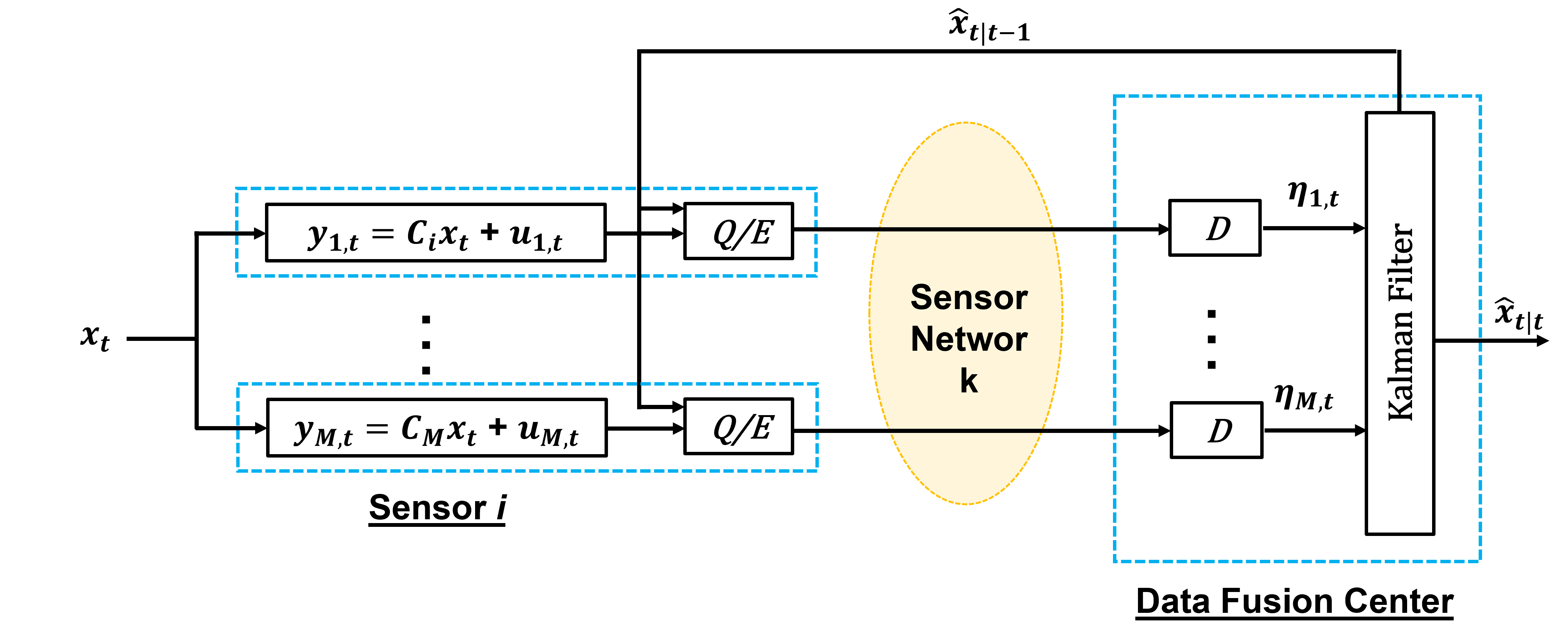}
    \caption{Data fusion center and distributed sensors. ``Q/E'' represents the quantizer/encoder module whereas ``D'' represents  the decoder module.}
    \label{fig:network}
\end{figure*}

\subsection{Entropy-coded dithered quantization}\label{subsec:quantDef}
We assume that sensor platforms discretize, and encode their measurements into binary codewords, using  \emph{entropy-coded dithered quantization} (ECDQ) \citen{derpich2012}. Define a uniform scalar quantizer with sensitivity $\Delta$ via 
\begin{align}
Q_{\Delta}&(z)=k\Delta 
\text{ if } (k-\tfrac{1}{2})\Delta\leq z < (k+\tfrac{1}{2})\Delta.
\end{align} Essentially, the function $Q_{\Delta}$ rounds its input to the nearest multiple of $\Delta$. \textit{Dithering} introduces intentional randomness into the quantization process to make the quantization error tractable. Let $\mathbf{z}$ be a random variable that we would like to encode, and let $\boldxi\sim\text{unif}[-\frac{\Delta}{2},\frac{\Delta}{2}]$ independent of $\mathbf{z}$. Define the \textit{quantization}
\begin{align}\label{eq:quant}
    \mathbf{q} = Q_{\Delta}&(\mathbf{z}+\boldxi).
\end{align} In general, $\mathbf{q}$ is a discrete random variable with countable support. Define the \textit{reconstruction} 
\begin{align}\label{eq:recon}
    \boldeta =  \mathbf{q}-\boldxi. 
\end{align} It can be shown (cf. \citen{zamir1992},  \citen{tanaka2016}) that 
\begin{align}\label{eq:dith}
    \boldeta  = \mathbf{z}+\mathbf{v}
\end{align} where $\mathbf{v}\sim\text{unif}[-\frac{\Delta}{2},\frac{\Delta}{2}]$ and independent of $\mathbf{z}$. Note that the \textit{same} realization of the dither signal $\boldxi$ is used in the quantization and reconstruction steps. 

\begin{figure*}[t]
    \centering
    \includegraphics[width=1\textwidth]{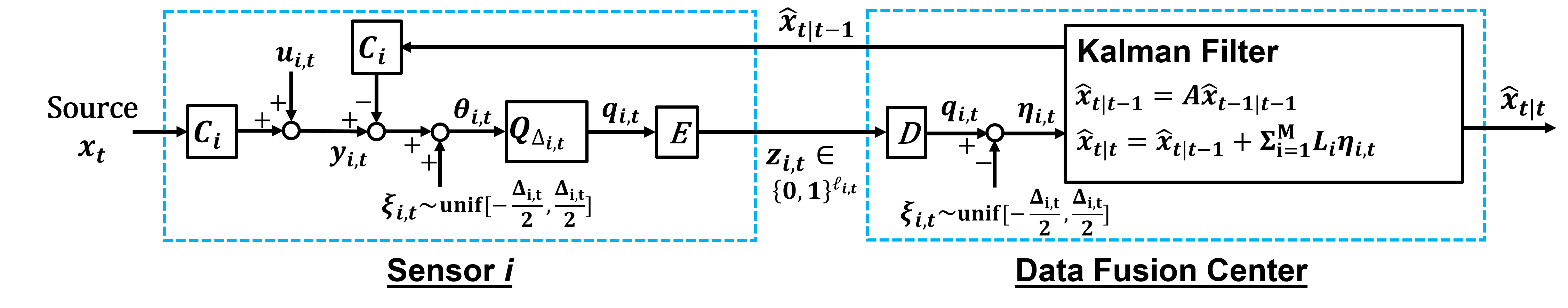}
    \caption{This figure gives an overview of the communication architecture between an individual sensor platform and the fusion center. Note that the dither signal, $\boldxi_{i,t}$ is assumed to be known at both the sensor platform and the fusion center. In practice, this “shared randomness'' can be accomplished by using synchronized pseudorandom number generators between each sensor platform and the fusion center.}
    \label{fig:channel1}
\end{figure*}

Fig.~\ref{fig:channel1} illustrates the use of dithering in our communication network model. We assume that each sensor platform shares a common dither signal with fusion center. At time $t$, denote the $i^{\mathrm{th}}$ platform's dither signal $\boldxi_{i,t}$. The stochastic process (defined over time and platforms) $\{\boldxi_{i,t}\}_{i\in1:M,t\in \mathbb{N}^+}$ consists of mutually independent uniform random variables 
\begin{align}
    \boldxi_{i,t}\sim \text{unif}[-\frac{\Delta_{i,t}}{2}, \frac{\Delta_{i,t}}{2}]
\end{align} where the sensitivities $\Delta_{i,t}$ will be designed in the sequel. Define the dither processes to be independent of $\{\mathbf{x}_{t},\mathbf{w}_t,\mathbf{u}_{1:M,t}\}_{t\in \mathbb{N}_{+}}$. At time $t$, platform $i$ computes the Kalman innovation corresponding to its measurement via
\begin{align}
\boldtheta_{i,t}=\by_{i,t}-C_i\hat{\bx}_{t|t-1}. 
\end{align} It then computes the dithered quantization \begin{align*}
\mathbf{q}_{i,t} = Q_{\Delta_{i,t}}&(\boldtheta_{i,t}+\boldxi_{i,t}).
\end{align*} The sensor platform uses \textit{lossless} variable-length \textit{entropy coding} to encode $\mathbf{q}_{i,t}$ into a finite-length binary string. Denote this \textit{codeword} $\bz_{i,t}$. The length of $\bz_{i,t}$ is a random variable, denote $\boldell_{i,t}$. It can be shown that there exists a lossless source code, adapted to the conditional probability mass function of $\mathbf{q}_{i,t}$ given $\boldxi_{i,t}$, with expected codeword length satisfying  (cf. [\citem{cover2012elements}, Theorem 5.5.1 and Exercise 5.28])
\begin{align}\label{eq:entropy} 
H(\mathbf{q}_{i,t}|\boldxi_{i,t})\leq \mathbb{E}(\boldell_{i,t}) < H(\mathbf{q}_{i,t}|\boldxi_{i,t})+1.
\end{align} Alternative, more tractable bounds on codeword length are developed in the next section. Since the binary communication channel from the sensor platform to the fusion center is assumed to be reliable the sensor platform receives $\bz_{i,t}$ exactly. Since the entropy code is lossless, the decoder at the fusion center recovers $\mathbf{q}_{i,t}$ exactly. The decoder then computes the dithered reconstruction
\begin{align}
    \boldeta_{i,t}=\mathbf{q}_{i,t}-\boldxi_{i,t}. 
\end{align} Using (\ref{eq:quant},\ref{eq:recon},\ref{eq:dith}) we have 
\begin{align}\label{eq:last22}
    \boldeta_{i,t}=\boldtheta_{i,t}+\bv_{i,t}\text{, where }\bv_{i,t}\sim \text{unif}[-\frac{\Delta_{i,t}}{2}, \frac{\Delta_{i,t}}{2}].
\end{align} It can be shown that the random variables  $\{\bv_{i,t}\}_{i\in1:M,t\in\mathbb{N}^+}$ are mutually independent and independent of $\{\boldtheta_{i,t}\}_{i\in1:M,t\in\mathbb{N}^+}$ \citen{tanaka2016}\citen{zamir1992}. An end-to-end model of entropy-coded dithered quantization is shown in Fig~\ref{fig:channel2}. The joint distributions of the random variables in Fig~\ref{fig:channel2} are identical to those in Fig~\ref{fig:channel1}.  

\begin{figure*}
    \centering
    \includegraphics[width=1\textwidth]{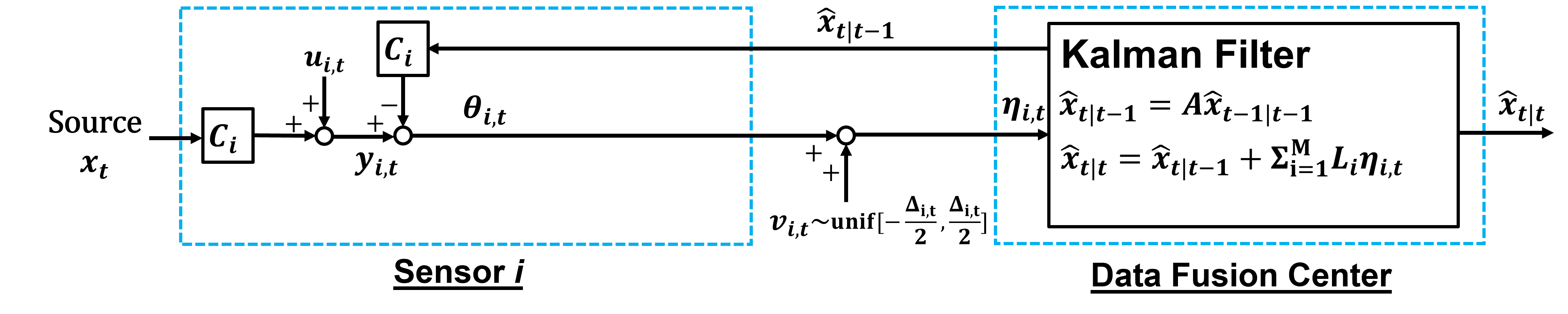}
    \caption{The net effect of the ECDQ scheme is that, from each sensor platform, the fusion center receives a linear measurement corresponding to the platform's Kalman innovation corrupted by additive \textit{uniform} noise. }
    \label{fig:channel2}
\end{figure*}

\begin{figure*}
    \centering
    \includegraphics[width=1\textwidth]{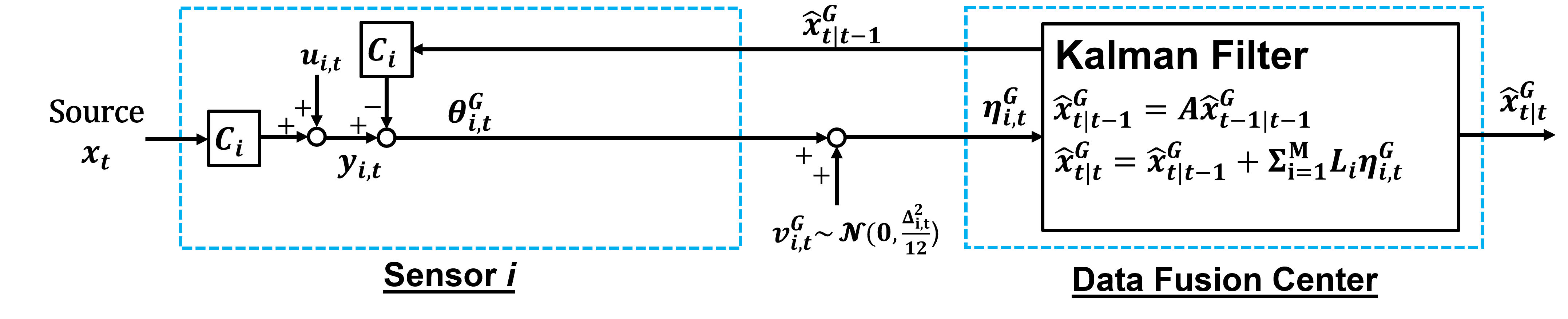}
    \caption{A ``jointly Gaussian'' version of the architecture shown in Fig.~\ref{fig:channel2}. The uniform noise in Fig.~\ref{fig:channel2} is replaced with Gaussian noise with the same first and second-order statistics. The process defined by the random variables in this figure is jointly Gaussian and has identical first and second moments to the process defined by the random variables in Fig.~\ref{fig:channel2}. }
    \label{fig:channel3}
\end{figure*}

\subsection{Communication cost approximation}\label{subsec:comcosap}
We define the communication cost associated with the $i^{\mathrm{th}}$ sensor over some time horizon $T$ as the time-averaged expected codeword length, e.g.
\begin{align}\label{eq:firstcost}
    R_i=\frac{1}{T}\sum_{t=1}^T \mathbb{E}(\boldell_{i,t}).
\end{align} The communication cost (\ref{eq:firstcost}) can be interpreted as the time-average rate of reliable (e.g. error-free) communication required to convey the quantized measurements $\{\mathbf{q}_{i,t}\}_{t=1}^{T}$ (defined in Section \ref{subsec:quantDef}) from the $i^\mathrm{th}$ sensor to the fusion center.  Under several reasonable architectures from wireless communications, this cost is an effective surrogate for the (expected) amount of physical layer resources consumed by the $i^{\mathrm{th}}$ platform on the uplink (cf. Section \ref{ssec:phy}). Equation (\ref{eq:entropy}) gave a bound on $\mathbb{E}(\boldell_{i,t})$ in terms of a conditional entropy. In the following lemma, we express these bounds in terms of a mutual information. 
\begin{lemma}
\label{lem:rate}For every sensor $i$ and time step $t$, the expected codeword length satisfies 
\[
I(\boldtheta_{i,t};\boldeta_{i,t})\leq \mathbb{E}(\boldell_{i,t})<I(\boldtheta_{i,t};\boldeta_{i,t})+1
\]
where the mutual information is taken with respect to the joint distribution of $\boldtheta_{i,t}$ and $\boldeta_{i,t}$ induced by the model in Fig.~\ref{fig:channel2}.
\begin{proof}
 See Appendix A.
\end{proof}
\end{lemma}

Note that $\boldtheta_{i,t}$ and $\boldeta_{i,t}$ are non-Gaussian. Directly evaluating the mutual information $I(\boldtheta_{i,t};\boldeta_{i,t})$ is rather difficult. In \citen{silva2009}, an approach to evaluating this mutual information was proposed via defining a ``jointly Gaussian'' version of the architecture in Fig.~\ref{fig:channel2}, shown here in Fig.~\ref{fig:channel3}. Define \begin{align}
    V_{i,t}  = \frac{\Delta^2_{i,t}}{12}.
\end{align} Essentially, the uniform quantization noise $\bv_{i,t}$ in Fig.~\ref{fig:channel2} is replaced with additive Gaussian noise, independent over sensors and time, with the same means and covariances, i.e. 
\begin{align}
    \bv_{i,t}^\text{G}\sim\mathcal{N}(0,V_{i,t})\text{, independent over $t$, $i$, and independent of $\{\mathbf{x}_{t}\}_{t\in\mathbb{N}^+}$}. 
\end{align} Under this model, the measurements and reconstructions are jointly Gaussian, and denoted $\boldtheta_{i,t}^\text{G}$ and $\boldeta_{i,t}^\text{G}$ respectively. Assume that the dynamical systems in Fig.~\ref{fig:channel2} and Fig.~\ref{fig:channel3} have initial conditions with identical first and second moments. It follows from the linearity of the models in Fig.~\ref{fig:channel2} and Fig.~\ref{fig:channel3}, and in particular the fact that the Kalman filter recursions are identical, that the stochastic processes $\{\bx_{t|t}^\text{G}, \bx_{t|t-1}^\text{G}, \boldtheta_t^\text{G}, \boldeta_t^\text{G}\}_{t=1,2,..., T}$ and $\{\bx_{t|t}, \bx_{t|t-1}, \boldtheta_t, \boldeta_t\}_{t=1,2,..., T}$ have identical first and second moments. 

The next lemma shows that  $I(\boldtheta_{i,t}; \boldeta_{i,t})$ can be bounded to within two bits of $I(\boldtheta_{i,t}^\text{G}; \boldeta_{i,t}^\text{G})$.
\begin{lemma}
\label{lem:silva}
\[
I(\boldtheta_{i,t}^\text{G}; \boldeta_{i,t}^\text{G})\leq 
I(\boldtheta_{i,t}; \boldeta_{i,t})<I(\boldtheta_{i,t}^\text{G}; \boldeta_{i,t}^\text{G})+\frac{1}{2}\log\frac{2\pi e}{12}.
\]
\end{lemma}
\begin{proof}
 See Appendix B.
\end{proof}

Using Lemma~\ref{lem:silva}, we can generate a new bound from Lemma~\ref{lem:rate} via
\begin{equation}
\label{eq:rate_approx}
I(\boldtheta_{i,t}^\text{G}; \boldeta_{i,t}^\text{G}) \leq \mathbb{E}(\boldell_{i,t})<I(\boldtheta_{i,t}^\text{G}; \boldeta_{i,t}^\text{G})+\underbrace{1+\frac{1}{2}\log\frac{2\pi e}{12}}_{\approx 1.254 \text{[bits]}}.
\end{equation} We will use the expression for $I(\boldtheta_{i,t}^\text{G}; \boldeta_{i,t}^\text{G})$ to approximate the communication cost in (\ref{eq:firstcost}). In particular we define $R^a_{i}$ as the \textit{approximate} time-averaged expected bitrate for the $i^{\mathrm{th}}$ sensor via
\begin{align}\label{eq:approx}
    R^a_{i} =\frac{1}{T}\sum_{t=1}^{T}I(\boldtheta_{i,t}^\text{G}; \boldeta_{i,t}^\text{G})
\end{align} Since $\bv_{i,t}^\text{G}\sim\mathcal{N}(0,\Delta_{i,t}^2/12)$ , the mutual information $I(\boldtheta_{i,t}^\text{G}; \boldeta_{i,t}^\text{G})$ is a function of the quantizer sensitivity $\Delta_{i,t}$. Tuning the set of sensitives allows us to adjust the data rates allocated to different sensors over time. We demonstrate this explicitly in Section \ref{sec:doc}. In the remainder of this section, we will use some properties from Kalman filters to extend our notion of ``second-order equivalence'' between the random variables in Fig.~\ref{fig:channel2} and  Fig.~\ref{fig:channel3} to the case of multiple sensors, and we 
define the sensor rate allocation problem. 
\subsection{Linear minimum mean square error estimation}
We consider sensor fusion in the setting that where the fusion center's observations follow the models of either Fig.~\ref{fig:channel1} or Fig.~\ref{fig:channel3}. In both cases, the fusion center uses a Kalman filter to recursively estimate $\mathbf{x}_{t}$. When all sensors conform to the non-Gaussian model of Fig.~\ref{fig:channel1}, at time $t$, the Kalman filter computes
\begin{subequations}
\begin{align}
    \mathbf{\hat{x}}_{t|t-1} = \text{the LMMSE estimate of $\bx_{t}$ given $\{\boldeta_{i,k}\}_{i\in 1:M, k\in 1:t-1}$ } 
\end{align} and
\begin{align}
    \mathbf{\hat{x}}_{t|t} = \text{the LMMSE estimate of $\bx_{t}$ given $\{\boldeta_{i,k}\}_{i\in 1:M, k\in 1:t}$ }.
\end{align}
\end{subequations} Conversely, in the case where all sensors conform to the Gaussian model (cf. Fig.~\ref{fig:channel3}), the LMMSE estimate computed by the Kalman filter corresponds to the minimum mean square error (MMSE) estimate, e.g., 
\begin{subequations}
\begin{align}
    \hat{\bx}_{t|t-1}^\text{G}=\mathbb{E}(\bx_t|\boldeta_{1:t-1}^\text{G})
\end{align} and
\begin{align}
    \hat{\bx}_{t|t}^\text{G}=\mathbb{E}(\bx_t|\boldeta_{1:t}^\text{G}). 
\end{align}
\end{subequations} Analogously to the previous section, if the initial conditions in both the Gaussian and non-Gaussian systems have identical means and covariance matrices, the Kalman filter recursions, in particular the achieved mean squared errors, are identical. Thus, for any fixed choice of quantizer sensitives $\Delta_{i,t}$ the estimator in Fig.~\ref{fig:channel1}  (equivalently,  Fig.~\ref{fig:channel2}) achieves the same MSE performance as the estimator in  Fig.~\ref{fig:channel3}. 


Let $U$ be the diagonal matrix defined in (\ref{eq:umatdef}) and let $V_{t}$ be a diagonal matrix such with the $V_{i,t}$s as diagonal entries i.e.
\begin{align}\label{eq:vtdef}
V_{t}=\begin{bmatrix}
    \frac{\Delta_{1,t}^2}{12} & 0 & \dots & 0 \\
    0 & \frac{\Delta_{2,t}^2}{12} & \dots & 0 \\
    \vdots & \vdots & \ddots & \vdots \\
    0 & 0 & \dots & \frac{\Delta_{M,t}^2}{12} 
\end{bmatrix}.
\end{align}
The Kalman gain at time $t$ is given by the forward Riccati recursion via
\begin{equation}
\label{eq:kalmangain}
L_t=P_{t|t-1}C^\top(CP_{t|t-1}C^\top+U+V_{t})^{-1}.
\end{equation} Define $P_{t|t}\in\mathbb{S}_{++}^n$ and $P_{t+1|t}\in\mathbb{S}_{++}^n$ as the estimation error covariances $P_{t|t}:=\text{Cov}(\bx_t-\hat{\bx}_{t|t}^\text{G})$ and
$P_{t+1|t}:=\text{Cov}(\bx_{t+1}-\hat{\bx}_{t+1|t}^\text{G})
$. 
These matrices satisfy the recursion
\begin{subequations}
\label{eq:cov}
\begin{align}
P_{t|t}^{-1}&=P_{t|t-1}^{-1}+C^\top (U+V_t)^{-1} C \\
P_{t+1|t}&=AP_{t|t}A^\top+FF^\top.
\end{align}
\end{subequations} We reiterate that due to the ``second-order equivalence'' between the Gaussian and non-Gaussian systems, we have for all $t$ that $\text{Cov}(\bx_t-\hat{\bx}_{t|t}^\text{G})=P_{t|t}$ and $\text{Cov}(\bx_{t+1}-\hat{\bx}_{t+1|t}^\text{G})=P_{t+1|t}$.

\subsection{The Sensor Rate Allocation (SRA) Problem}\label{subsec:sradef} 
We now propose the optimization problem to be studied in the remainder of this paper. We propose to minimize a weighted average of the approximate data rates $R^a_i$ assigned to each sensor $i=1, 2, ... , M$ subject to the constraint that the MSE performance of the fusion center's estimator does not exceed $\beta$. Define a set of positive weights  $\alpha_{i}$, $i\in1:M$ to represent the cost of transmitting a single bit from platform $i$ to the fusion center. The total communication cost is $\sum_{i=1}^M \alpha_iR_{i}$, where $R_{i}$ is defined in (\ref{eq:firstcost}). The approximate cost,  defined in terms of the approximate bitrates (cf. (\ref{eq:approx})), is defined as $\sum_{i=1}^M \alpha_iR^a_{i}$.  Our decision variables are the sensor platforms' quantizer sensitivities, e.g. $\{\Delta_{i,t}\}_{i\in 1:M, t\in1:T}$. Via (\ref{eq:vtdef}), it can be seen that for any time $t$ the set of feasible quantizer sensitivities are in one-to-one correspondence with the set of diagonal, positive semidefinite matrices. Let the set of positive semidefinite diagonal $M\times M$ matrices be denoted $\mathbb{D}^M_+$.
For $\beta$ the constraint on the fusion center's MSE, we define the sensor rate allocation problem as 
\begin{subequations}
\label{eq:s3ra}
\begin{align}
\underset{\substack{V_{t}\in \mathbb{D}^M_+\\t \in 1:T}}{\min}\quad & \frac{1}{T}\sum_{t=1}^T\sum_{i=1}^M \alpha_i I(\boldtheta_{i,t}^{\text{G}};\boldeta_{i,t}^{\text{G}})  \label{eq:s3ra_a} \\ 
\text{s.t.}\quad &\quad  \frac{1}{T} \sum_{t=1}^T \mathbb{E} \|\bx_t-\hat{\bx}_{t|t}\|^2 \leq \beta. \label{eq:s3ra_b}
\end{align}
\end{subequations} 
Say $\{V_t^*\}_{t=1}^T$ is a feasible solution of \eqref{eq:s3ra} and $f^*$ is the attained value. Choosing the quantization sensitivities such that $\frac{\Delta_{i,t}^2}{12}=V_{i,t}^*$ guarantees a communication cost satisfying 
\begin{align}
    \sum_{i=1}^M \alpha_iR_{i} \le f^*+(1+\frac{1}{2}\log\frac{2\pi e}{12})\times\sum_{i=1}^M \alpha_i.
\end{align}

We also consider the infinite horizon problem where the quantizer sensitivities are time invariant, namely 
\begin{subequations}
\label{eq:s3ra_inf}
\begin{align}
\underset{\substack{V\in \mathbb{D}^M_+\\t \in 1:T}}{\min}\quad & \limsup_{T\rightarrow \infty}\frac{1}{T}\sum_{t=1}^T\sum_{i=1}^M \alpha_i I(\boldtheta_{i,t}^{\text{G}};\boldeta_{i,t}^{\text{G}}) \label{eq:s3ra_inf_a}\\ 
\text{s.t.}\quad &\quad  \limsup_{T\rightarrow \infty}\frac{1}{T} \sum_{t=1}^T \mathbb{E} \|\bx_t-\hat{\bx}_{t|t}\|^2 \leq \beta. \label{eq:s3ra_inf_b}
\end{align}
\end{subequations} Essentially, the optimal solution to the problem (\ref{eq:s3ra_inf}) can be interpreted as the best time-invariant solution to the infinite-horizon SRA problem. The analysis of time-varying solutions to the SRA problem presents an opportunity for future work (cf. Section \ref{sec:conclusion}). In the following subsection, we describe an application of the sensor rate allocation optimization to resource management in remote sensing scenario.

\subsection{Application example: physical layer resource allocation}\label{ssec:phy}
The minimizations in (\ref{eq:s3ra}) and (\ref{eq:s3ra_inf}) can be applied to a variety of resource allocation problems in wireless (and, in principle, wireline) communication systems. By choosing the weights $\alpha_i$ in (\ref{eq:s3ra_inf}) appropriately, we can design a data compression scheme that seeks to minimize the number of physical layer resources required to achieve the required estimator performance at the fusion center. As an example, we consider a setup similar to the one proposed in \citen{rebal}, where several remote sensing platforms communicate their observations to a data fusion center over a shared wireless medium. We present a simple example to conclude this section. 

\begin{figure}
    \centering
    \includegraphics[scale=0.19]{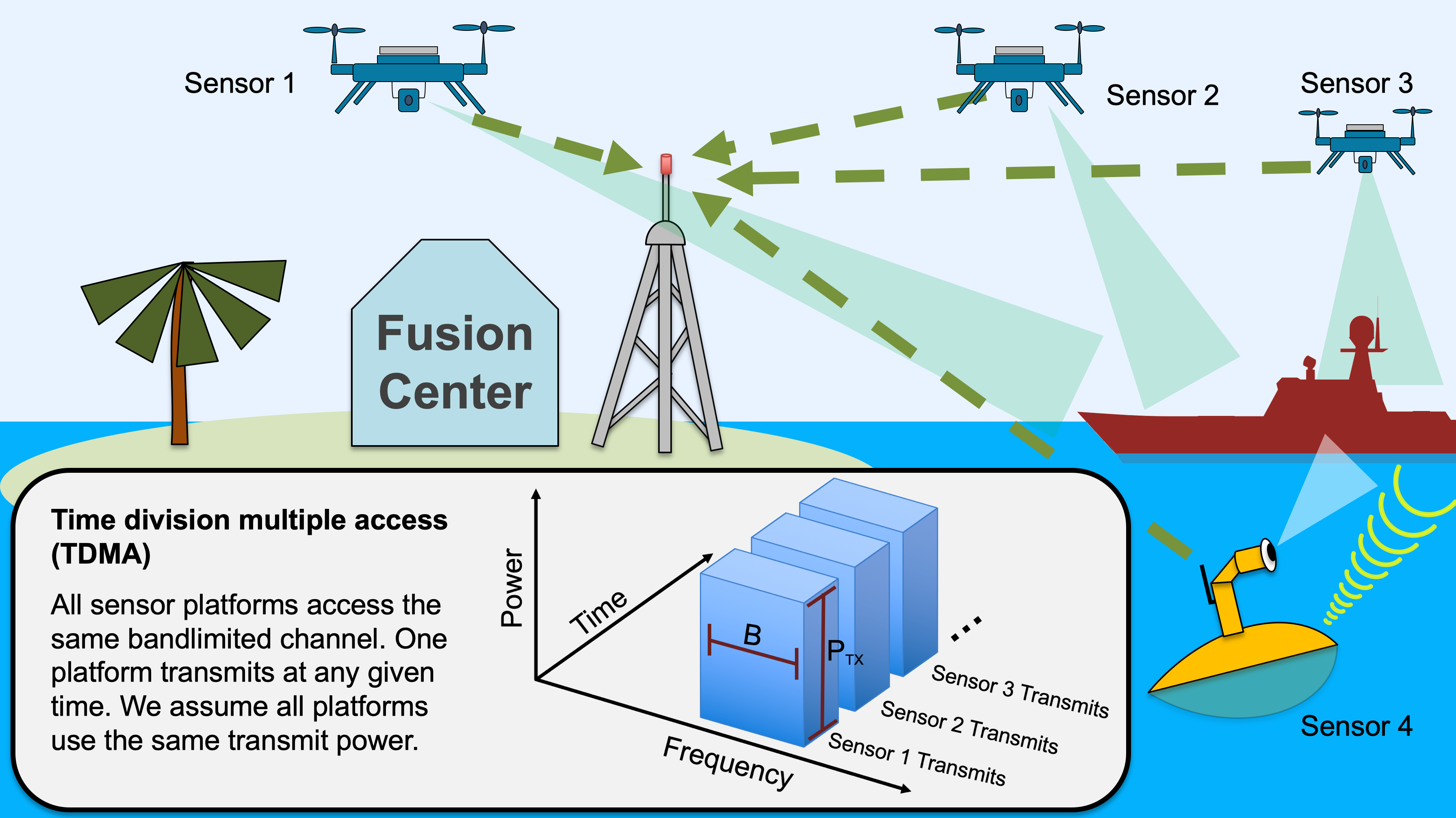}
    \caption{In this example, the different sensor platforms relay their measurements to the fusion center over a common wireless link using TDMA (see inset).}
    \label{fig:tdma_phy}
\end{figure}
Consider the case where several remote sensors, using a fixed uplink transmit power $P_{\mathrm{UL}}$ (in Watts), communicate their measurements to the fusion center via  time-division multiple access (TDMA) over a channel of passband bandwidth $B$ (in Hz). In such a system, the platforms take turns transmitting their measurements over the same bandlimited channel---see Figure \ref{fig:tdma_phy}. We assume that the $i^{\mathrm{th}}$ sensor platform is at a range of $r_i$ from the fusion center, that all sensor platforms have identical transmission systems, and that the channel from each sensor to the fusion center is line-of-sight. The amount of power received from transmitter $i$ depends on range, as well as a variety of factors, like the gains from the antenna systems and the losses from other components. We let $L$ encapsulate the factors that do not depend on the range---in the present case $L$ is assumed the same for all transmitters. The additive noise at the receiver is assumed to be white and Gaussian with power spectral density $N_{0}$. Since only one sensor platform transmits at any given time,  the effective channel between the $i^{\mathrm{th}}$ transmitter and the fusion center is a point-to-point additive white Gaussian noise (AWGN) channel with a signal-to-noise ratio (SNR) of
\begin{align}
    \mathrm{SNR}_{i} = L\frac{P_{\mathrm{UL}}}{r^2_{i}BN_{0}}.
\end{align}  The \textit{capacity} of the wireless channel from receiver $i$ to the fusion center is
\begin{align}
    C_{i} = B\log_2(1+\mathrm{SNR}_{i})\text{ bits/second},
\end{align} which is the maximum reliable (e.g. error-free) data-rate that can be sustained from the $i^{\mathrm{th}}$ transmitter to the receiver. The \textit{airtime}, $\tau_i$, in seconds, required to send $b_{i}$ bits from the $i^{\mathrm{th}}$ sensor platform to the fusion center is lower bounded via
\begin{align}
    \tau_i \ge \frac{b_{i}}{C_{i}}.
\end{align} This bound is in general loose, as the zero-error channel capacity is achievable only asymptotically (c.f. \citen{polyanskiy}). However, with modern error-correcting codes (including low-density parity-check and turbo codes) it is possible to communicate with very small error probability near the capacity for the point-to-point AWGN channel\citen{procCap}. We thus make the assumption that
\begin{align}
    \tau_i\approx \frac{b_{i}}{C_{i}}. 
\end{align} In particular, the airtime is directly proportional to the energy expended on communication. Under the assumption that all the sensor platforms use the same transmit power, the energy expended to transmit $b_{i}$ bits is $E_{i} = P_{\mathrm{UL}}\tau_i$.

If the system uses time-division duplexing in addition to TDMA, the downlink communication also time-shares the same frequency channel shared between the uplink users. Let $N_{i}$ denote the noise power spectral density of the channel from the fusion center to the $i^\mathrm{th}$ user, and let $\theta^2_{\max} = \max_{i}N_{i}r_i^2$. 
Assuming line-of-sight communication, and that the fusion center feeds back a message containing $b_{\mathrm{DL}}$ bits to all of the remote platforms, the airtime used by the downlink can be approximated by 
\begin{align}
    \tau_\mathrm{DL} \approx \frac{b_{\mathrm{DL}}}{B\log_2(1+LP_{\mathrm{DL}}/(B\theta^2_{\max}))}.
\end{align} Note that for a fixed downlink message size $b_{\mathrm{DL}}$, $\tau_\mathrm{DL}$ decreases to $0$ as $P_{\mathrm{DL}}$ increases. Thus, we claim that in a setting where the uplink is subject to a strict power constraint (as appropriate, for example, if the users are battery-powered) and the downlink power constraint, $P_{\mathrm{DL}}$, is sufficiently high it is reasonable to assume that the fusion center can feed back a finely quantized version $\hat{\mathbf{x}}_{t|t-1}$ to the users using a negligible airtime.

Minimizing the uplink airtime of a sensing task has a great deal of utility. It minimizes the amount of energy the network spends on uplink communication. Furthermore, completing a sensing task with the minimum airtime is useful in and of itself, as it may allow other users and systems the opportunity to access the channel without causing (or experiencing) additional interference. If we make the choice
\begin{align}
    \alpha_{i} = \frac{1}{C_{i}},
\end{align} the optimizations in (\ref{eq:s3ra}) can be seen as minimizing the expected airtime (and, consequently, the energy expended on communication) required to complete the sensing task in the finite horizon.  Note that the expected airtime expended by each sensor can vary over sampling periods, since at sampling period $t$ sensor $i$ will send an average of about $\lceil I(\bx_t;\boldeta_{i,t}^{\text{G}}|\boldeta_{1:t-1}^{\text{G}})+1.254 \rceil$ bits, where the mutual information is computed under the probability measure imposed by the minimizing policy. 

It should be noted that, due to our use of variable-length coding, the required airtime depends on the length of the codeword and is thus a random variable. In delay-tolerant applications, this may not be much of a problem. However, assume the sample rate of the sensors is fixed at $T_{s}$. The network will become congested if, during any sampling period, the required airtime exceeds $T_{s}$. In this setting, new measurements are being produced at a rate faster than the communication delay---more complicated scheduling algorithms may be required to ensure that the fusion center estimator achieves the desired performance by any real-time deadline (see, for example, \citen{rebal}). If the expected airtime, per sampling period, obtained from the optimization (\ref{eq:s3ra}) is small compared to $T_{s}$, the Markov inequality can be used to derive a simple bound on the congestion probability. These ideas can be extended to more general types of communication channels (e.g. fading channels) and corresponding more general notions of capacity.

\section{Conversion to difference-of-convex problem}\label{sec:doc}
In this section, we reformulate  (\ref{eq:s3ra}) and (\ref{eq:s3ra_inf}) as difference of convex (DC) optimization problems. 


\subsection{Reformulation of mutual information}

The mutual information \eqref{eq:s3ra_a} can be expressed as:
\begin{equation}
\label{eq:dir_info}
    \begin{split}
        I(\boldtheta_{i,t}^\text{G}; \boldeta_{i,t}^\text{G}) &=h(\boldeta_{i,t}^\text{G})-h(\boldeta_{i,t}^\text{G}|\boldtheta_{i,t}^\text{G})\\
     &= h(C_i(\bx_t-\hat{\bx}_{t|t-1}^\text{G})+\bu_{i,t}+\bv^{\text{G}}_{i,t})-h(\boldtheta_{i,t}^\text{G}+\bv^{\text{G}}_{i,t}|\boldtheta_{i,t}^\text{G}) 
    \\
    &= \tfrac{1}{2}\log(C_i P_{t|t-1}C_i^\top+U_i+V_{i,t})-  \tfrac{1}{2}\log(V_{i,t}).
    \end{split}
\end{equation}
Introducing the new variables $Q_{t|t-1}:=P_{t|t-1}^{-1}$, $Q_{t|t}:=P_{t|t}^{-1}$ and 
$\delta_{i,t}:= (U_i+V_{i,t})^{-1}$, we can rewrite the expression derived for $I(\boldtheta_{i,t}^\text{G}; \boldeta_{i,t}^\text{G})$ in \eqref{eq:dir_info} as

\begin{IEEEeqnarray}{rCl}
I(\boldtheta_{i,t}^\text{G}; \boldeta_{i,t}^\text{G}) &=&-\tfrac{1}{2}\log(\delta_{i,t}^{-1}-U_{i})-\tfrac{1}{2} \log{( \delta_{i,t}^{-1}+C_i Q_{t|t-1}^{-1}C_i^\top)^{-1}} \nonumber \\
&=&
\begin{cases}\min_{ \psi_{i,t}, \; \gamma_{i,t}} & \tfrac{1}{2} \log\psi_{i,t}-\tfrac{1}{2} \log \gamma_{i,t}  \\
\text{s.t.} & \psi_{i,t} \geq (\delta_{i,t}^{-1}-U_i)^{-1},\\
&\gamma_{i,t} \leq (\delta_{i,t}^{-1}+C_iQ_{t|t-1}^{-1}C_i^\top)^{-1}
\end{cases} \label{eq:nminus} \\
& {=}&
\begin{cases}\min_{\psi_{i,t}, \gamma_{i,t}} & \quad \tfrac{1}{2} \log\psi_{i,t}-\tfrac{1}{2} \log \gamma_{i,t} \\
\text{s.t.}
& 
\begin{bmatrix}
\delta_{i,t} & \delta_{i,t} U_i & \delta_{i,t}\\
U_i\delta_{i,t} & U_i & 0\\
\delta_{i,t}&0 &  \psi_{i,t} 
\end{bmatrix}\succeq 0,\\
&
\begin{bmatrix}
\delta_{i,t}-\gamma_{i,t} & \delta_{i,t}C_i \\
C_i^\top \delta_{i,t} & Q_{t|t-1}+C_i^\top \delta_{i,t} C_i
\end{bmatrix} \succeq 0.
\end{cases}
\label{eq:mi_lmi_a}
\end{IEEEeqnarray}
The slack variables $\psi_{i,t}$ and $\gamma_{i,t}$ were introduced in (\ref{eq:nminus}). To obtain the first inequality constraint in (\ref{eq:mi_lmi_a}), the Schur complement lemma is applied to the first constraint in (\ref{eq:nminus}). Likewise, the second inequality constraint in (\ref{eq:mi_lmi_a}) follows from expanding the second inequality in  (\ref{eq:nminus}) before applying the Shur complement lemma. 

\subsection{Reformulation of mean square error}
Using the variable $Q_{t|t}$, the fusion center's estimator MSE \eqref{eq:s3ra_b} can be written as

\begin{equation}
\label{eq:mse_lmi}
    \mathbb{E}\|\bx_t-\hat{\bx}_{t|t}^{\text{G}}\|^2=\text{Tr}(P_{t|t})=\text{Tr}(Q_{t|t}^{-1}) = \begin{cases}\min_{S_t} & \text{Tr}(S_t)  \\
     \text{s.t.} & Q_{t|t}^{-1}\preceq S_t 
\end{cases}
=\begin{cases}\min_{S_t} & \text{Tr}(S_t) \\
  \text{s.t.} & \begin{bmatrix}
    S_t & I \\ I & Q_{t|t}
    \end{bmatrix}\succeq 0
    \end{cases}.
\end{equation}
The final step of (\ref{eq:mse_lmi}) follows from Schur complement formula. 

\subsection{Reformulation of the original problem (\ref{eq:s3ra})}
Using \eqref{eq:cov}, \eqref{eq:mi_lmi_a}, and \eqref{eq:mse_lmi}, problem \eqref{eq:s3ra} can be reformulated as:
\begin{subequations}
\label{eq:prob1}
\begin{align}
    \min \quad & \frac{1}{T}\sum_{t=1}^T \sum_{i=1}^M \frac{\alpha_i}{2}(\log \psi_{i,t}-\log\gamma_{i,t}) \label{eq:prob1_a}\\
    \label{eq:prob1_g}
    \text{s.t.} \quad &
    \begin{bmatrix}
     \delta_{i,t} & \delta_{i,t} U_i & \delta_{i,t}\\
U_i\delta_{i,t} & U_i & 0\\
\delta_{i,t}&0 &  \psi_{i,t} 
\end{bmatrix}\succeq 0,\\
&    
\begin{bmatrix}
    \delta_{i,t}-\gamma_{i,t} & \delta_{i,t}C_i \\
C_i^\top \delta_{i,t} & Q_{t|t-1}+C_i^\top \delta_{i,t} C_i
\end{bmatrix} \succeq 0,\label{eq:prob1_b} \\
&\begin{bmatrix}
    S_t & I \\ I & Q_{t|t}
    \end{bmatrix}\succeq 0, \;\; \frac{1}{T}\sum\nolimits_{t=1}^T \text{Tr}(S_t)\leq \beta, \label{eq:prob1_d}\\
&Q_{t|t}=Q_{t|t-1}+\sum\nolimits_{i=1}^M \delta_{i,t}C_i^\top C_i, \label{eq:prob1_e} \\
&Q_{t|t-1}^{-1}=AQ_{t-1|t-1}^{-1}A^\top+FF^\top, \label{eq:prob1_f}
\end{align}
\end{subequations}
where  the decision variables are $[\psi_{i,t},\delta_{i,t}, \gamma_{i,t}]$, for $i=1, ... , M$ and $t=1, ... , T$,  $S_t$ for $t=1, ... , T$, and $[Q_{t|t}, Q_{t|t-1}]$ for $t=2, ... , T$. The constraints \eqref{eq:prob1_e} and \eqref{eq:prob1_f} are imposed for $t=2, ... , T$, starting from $Q_{1|1}=P_{1|1}^{-1}$. All constraints in problem \eqref{eq:prob1} are convex except \eqref{eq:prob1_f}.
In the following proposition, we show the non-convex constraint \eqref{eq:prob1_f} can be replaced by a convex constraint without changing the nature of the problem. More precisely, we show that the equality in \eqref{eq:prob1_f}  can be replaced by an inequality and 
we can solve \eqref{eq:new_prob1}, instead of \eqref{eq:prob1}, to obtain the optimal rate allocation.
\begin{subequations}
\label{eq:new_prob1}
\begin{align}
    \min \quad & \frac{1}{T}\sum_{t=1}^T \sum_{i=1}^M \frac{\alpha_i}{2}(\log \psi_{i,t}-\log\gamma_{i,t}) \label{eq:new_prob1_a} \\ \label{eq:new_prob1_g}
    \text{s.t.} \quad &
\begin{bmatrix}
\delta_{i,t} & \delta_{i,t} U_i & \delta_{i,t}\\
U_i\delta_{i,t} & U_i & 0\\
\delta_{i,t}&0 &  \psi_{i,t} 
\end{bmatrix}\succeq 0,    
\\
&    
\begin{bmatrix}
    \delta_{i,t}-\gamma_{i,t} & \delta_{i,t}C_i \\
C_i^\top \delta_{i,t} & Q_{t|t-1}+C_i^\top \delta_{i,t} C_i
\end{bmatrix} \succeq 0, \label{eq:new_prob1_b} \\
&\begin{bmatrix}
    S_t & I \\ I & Q_{t|t}
\end{bmatrix}\succeq 0, \;\; \frac{1}{T}\sum\nolimits_{t=1}^T \text{Tr}(S_t)\leq \beta, \label{eq:new_prob1_d} \\
&Q_{t|t}=Q_{t|t-1}+\sum\nolimits_{i=1}^M \delta_{i,t}C_i^\top C_i, \label{eq:new_prob1_e} \\
&Q_{t|t-1}^{-1}\succeq AQ_{t-1|t-1}^{-1}A^\top+FF^\top \label{eq:new_prob1_f}
\end{align}
\end{subequations}
\begin{proposition}
\label{lem:ineq}
Let $[ \psi_{i,t}^{*}, \delta_{i,t}^{*}, \gamma_{i,t}^{*}, S_t^{*}, Q_{t|t}^{*}, Q_{t|t-1}^{*}]$ be the optimal solution for \eqref{eq:new_prob1}. Then, $[\psi_{i,t}^{*}, \delta_{i,t}^{*}, \gamma_{i,t}^{*}, S_t^{*}, Q_{t|t}^{**}, Q_{t|t-1}^{**}]$ is the optimal solution for \eqref{eq:prob1}, where $Q_{t|t}^{**}$ and $Q_{t|t-1}^{**}$ are calculated recursively by 
\begin{subequations}
\label{eq:definition}
\begin{align}
    &Q_{t|t-1}^{**-1} = AQ_{t-1|t-1}^{**-1}A^\top+FF^\top, \label{eq:definition_a}\\
    &Q_{t|t}^{**}=Q_{t|t-1}^{**}+\sum\nolimits_{i=1}^M \delta_{i,t}^{*}C_i^\top C_i, \label{eq:definition_b}
\end{align}
\end{subequations}
starting from $Q_{1|1}^{**} = Q_{1|1}^*$. Moreover, if we denote the optimal values of \eqref{eq:prob1} and \eqref{eq:new_prob1} by $J_1^*$ and $J_2^*$ respectively, we have $J_1^*=J_2^*$. 

\end{proposition}

\begin{proof}
See Appendix \ref{app4}.
\end{proof}

Constraint \eqref{eq:new_prob1_f} can be written as an equivalent linear matrix inequality (LMI) \citen{boyd1994}  condition.
Therefore, problem \eqref{eq:s3ra} can be written as:
\begin{subequations}
\label{eq:prob2}
\vspace{-3ex}
\begin{align}
    \min \quad & \frac{1}{T}\sum_{t=1}^T \sum_{i=1}^M \frac{\alpha_i}{2}(\log \psi_{i,t}-\log\gamma_{i,t}) \label{eq:prob2_a}\\ \label{eq:prob2_g}
    \text{s.t.} \quad &
    \begin{bmatrix}
\delta_{i,t} & \delta_{i,t} U_i & \delta_{i,t}\\
U_i\delta_{i,t} & U_i & 0\\
\delta_{i,t}&0 &  \psi_{i,t} 
\end{bmatrix}\succeq 0,
\\
&
\begin{bmatrix}
    \delta_{i,t}-\gamma_{i,t} & \delta_{i,t}C_i \\
C_i^\top \delta_{i,t} & Q_{t|t-1}+C_i^\top \delta_{i,t} C_i
\end{bmatrix} \succeq 0, \label{eq:prob2_b} \\
& \begin{bmatrix}
    S_t & I \\ I & Q_{t|t}
    \end{bmatrix}\succeq 0,\;\; \frac{1}{T}\sum\nolimits_{t=1}^T \text{Tr}(S_t)\leq \beta, \label{eq:prob2_d}\\
&Q_{t|t}=Q_{t|t-1}+\sum\nolimits_{i=1}^M \delta_{i,t}C_i^\top C_i. \label{eq:prob2_e} \\
&\begin{bmatrix}
Q_{t|t-1} & Q_{t|t-1}A & Q_{t|t-1}F \\
A^\top Q_{t|t-1} & Q_{t-1|t-1} & 0 \\
F^\top Q_{t|t-1} & 0 & I
\end{bmatrix}\succeq 0, \label{eq:prob2_f}.
\end{align}
\end{subequations}


It is simple to verify that \eqref{eq:prob2} is the minimization of the difference of two convex functions subject to convex constraints 
and thus it is an instance of a difference-of-convex program. The infinite horizon (time-invariant) counterpart of \eqref{eq:prob2} can be formulated as:
\begin{subequations}
\label{eq:prob3}
\begin{align}
    \min \quad &  \sum_{i=1}^M \frac{\alpha_i}{2}(\log \psi_{i}-\log\gamma_{i}) \\ \label{eq:prob3_e}
    \text{s.t.} \quad &
    \begin{bmatrix}
\delta_{i} & \delta_{i} U_i & \delta_{i}\\
U_i\delta_{i} & U_i & 0\\
\delta_{i}&0 &  \psi_{i} 
\end{bmatrix}\succeq 0, \; \forall i= 1, ... , M,
\\
&    
\begin{bmatrix}
    \delta_{i}-\gamma_{i} & \delta_{i}C_i \\
C_i^\top \delta_{i} & \hat{Q}+C_i^\top \delta_{i} C_i
\end{bmatrix} \succeq 0, \; \forall i= 1, ... , M, \label{eq:prob3_b} \\
& Q=\hat{Q}+\sum\nolimits_{i=1}^M \delta_{i}C_i^\top C_i, \;\; \text{Tr}(S)\leq \beta, \label{eq:prob3_c}\\
&  \begin{bmatrix}
    S & I \\ I & Q
    \end{bmatrix}\succeq 0, \;\; \begin{bmatrix}
\hat{Q} & \hat{Q}A & \hat{Q}F \\
A^\top \hat{Q} & Q & 0 \\
F^\top \hat{Q} & 0 & I
\end{bmatrix}\succeq 0, \label{eq:prob3_d}.
\end{align}
\end{subequations}

\section{Convex-concave procedure (CCP)}\label{sec:ccp}


In this section, we develop an iterative algorithm based on the convex-concave procedure (CCP) \citen{lipp2016} to find local minima of the time-varying and time-invariant sensor data-rate allocation problems (\ref{eq:prob2}) and (\ref{eq:prob3}). CCP starts by  over-approximating the non-convex terms of the DC program via linearization around a nominal point. The resulting convex problem can then be solved efficiently, and the algorithm iterates by convexifying the problem around the obtained solution. The process terminates when a local minimum is found.  Starting from a feasible nominal point, as it is shown in \citen{lipp2016}, all subsequent iterations will be feasible and the algorithm converges to a local minima.

We now show how the CCP algorithm operates for the time-invariant problem \eqref{eq:prob3}. 
For this problem, we linearize the concave term $\log\psi_i$ around the nominal point $\hat{\psi}_i$
which gives an upper bound as $\log\psi_i\leq \frac{1}{\hat{\psi}_i}(\psi_i-\hat{\psi}_i)+\log\hat{\psi}_i$. 
Thus, at each iteration of CCP algorithm we solve:
\begin{subequations}
\label{eq:prob4}
\begin{align}
    \min \quad &  \sum_{i=1}^M \frac{\alpha_i}{2}(\psi_i/\hat{\psi}_i-1+\log\hat{\psi}_i-\log\gamma_{i}) \\
    \text{s.t.} \quad & (\ref{eq:prob3_e})-(\ref{eq:prob3_d}). 
\end{align}
\end{subequations}

If we denote the solution of (\ref{eq:prob4}) at iteration $k$ by $(\psi_i^k,\delta_i^k,\gamma_{i}^k, S^k, Q^k, \hat{Q}^k)$, at iteration $k+1$ (\ref{eq:prob4}) is solved for $\hat{\psi}_i=\psi_i^k$. Note that the optimal value of \eqref{eq:prob4} is an upper bound for the optimal value of \eqref{eq:prob3} with same decision variables $S, Q, \hat{Q}, \psi_i, \delta_i,$ and $\gamma_i$ for $i=1,...M$. The CCP algorithm we use to perform the time-invariant optimization \eqref{eq:prob3} is outlined in Algorithm~\ref{alg1}. A similar algorithm can be used to solve (\ref{eq:prob2}). Problem \eqref{eq:prob4} is a log-det minimization and can be solved via  SDP solvers in polynominal time \citen{ben2001lectures}. As it will be demonstrated in Section~\ref{sec:simulations}, in our simulations Algorithm~\ref{alg1} tends to converge in a few tens of CCP iterations.

\begin{algorithm}\label{alg:ccp}
\caption{Convex-Concave Procedure (CCP)}\label{alg1}
\begin{algorithmic}
\State Set $tolerance$ sufficiently small;
\State Set initial value $\hat{\psi}_i \leftarrow 1$ for $i=1, 2, \dots, M$;
\For{$k=1, 2, \dots$}
\State solve (\ref{eq:prob4}).
\State $(\psi_i^k,\delta_i^k,\gamma_{i}^k, S^k, Q^k, \hat{Q}^k)\leftarrow$ Obtain optimal solution;
\State $f^{k}  \leftarrow$ Optimal value of \eqref{eq:prob4};
\State Update $\hat{\psi}_i\leftarrow \psi_i^k$ for $i=1:M$;
\If{$f^{k-1}-f^k \leq tolerance$}
\State break;
\EndIf
\EndFor
\end{algorithmic}
\end{algorithm}

\section{Numerical Analysis}\label{sec:simulations} 

\label{sec:Num}

In this section, we present two numerical simulations showing the effectiveness of Algorithm~\ref{alg1}. In both studies, we optimize the sum bitrate by setting $\alpha_{i}$ = 1 for all sensors.

\subsection{Heat diffusion system}
We first consider the problem of estimating the temperature distribution over a time-invariant 1D heat diffusion system \citen{chahlaoui2002}.
We consider a slender concrete rod that is $15$ m long. The length of the rod is equally divided by 60 nodes. The continuous-time heat diffusion system matrix $A$ is given by 
\begin{align*}
A_c = \frac{a}{h^2} \begin{bmatrix}
2 & -1 && \\
-1 & 2 & -1 &\\
&& \ddots & \ddots & \ddots &\\
&&& -1 & 2 & -1\\
&&&& -1 & 2
\end{bmatrix}
\end{align*}
where $a=7.5\times 10^{-7} \text{m}^2/\text{s}$ is the thermal diffusivity and $h = 0.2459$ m is the length of each segment. We assume the system is driven by a random heat input at each node. Let $E_{d} = I-\frac{\Delta_{t}}{2}A_c$ and $A_{d} = I+\frac{\Delta_{t}}{2}A_c$. We consider tracking the temperature of the rod at each node. A discrete-time model for this system of the form (\ref{eq:source}) is found by setting
\begin{subequations}
\begin{align}
    A = E_{d}^{-1}A_{d}\text{, and}
\end{align} 
\begin{align}
    F = I.
\end{align} 
\end{subequations} In these simulations, we set $\Delta_t =1$. Our set of sensor platforms is a thermometer placed at each node; we assume that the thermometers can measure the node's temperature noiselessly. Let $e_i$ for $i\in 1:60$ denote unit vectors from the standard basis. Formally, the measurement matrix for the $i^{\text{th}}$ sensor is 
\begin{align}
    C_i = e_ie_i^{\mathrm{T}}. 
\end{align} The measurement data is transmitted to a data fusion center over a communication network as described in Section \ref{sec:formulation}.  

Fig.~\ref{fig:CCP-ADMM01} shows the performance of Algorithm~\ref{alg1} as a function of the number of CCP iterations assuming that the constraint on the fusion center's MSE is $\beta$ = $0.1$. The algorithm is initialized with an equal data rate (quantizer sensitivity) assigned to each sensor. Fig.~\ref{fig:CCP-ADMM01}~(a) plots the sum data rate obtained by the solution obtained after a given number of CCP iterations, while Fig~\ref{fig:CCP-ADMM01}~(b) shows the data rates assigned to each sensor after each such iteration. In Fig~\ref{fig:CCP-ADMM01}~(b) the assigned data rate is presented as color-coded block at each iteration. Figs.~\ref{fig:CCP-ADMM01}~(a) and (b) demonstrate that both the total data rate for the network, and the data rates assigned to each sensor platform, converge after a sufficient number of CCP iterations. 


\begin{figure}
    \centering
    \includegraphics[width=0.9\columnwidth]{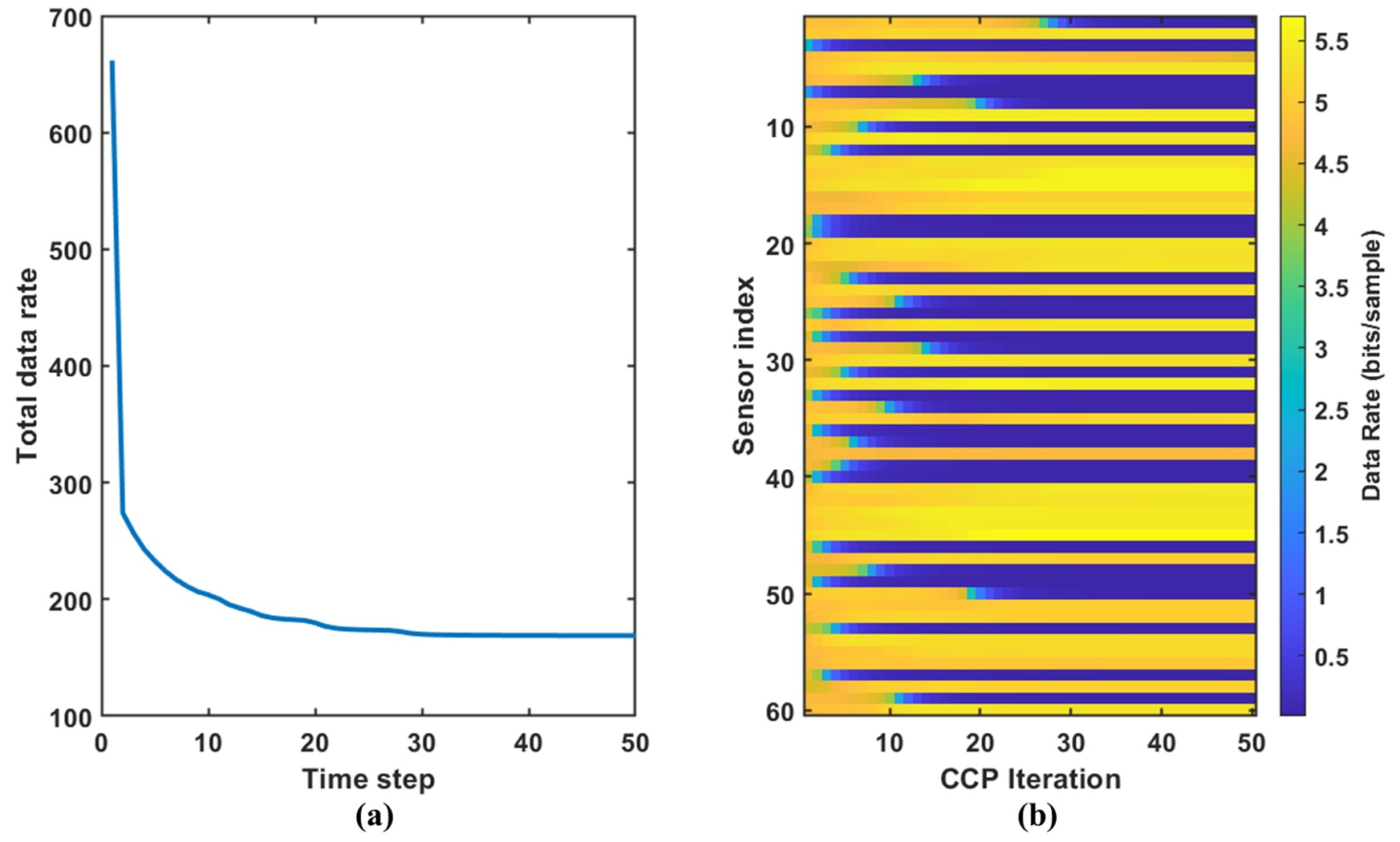}
    \caption{These plots show optimized data rates found after a given number of CCP iterations with a MSE constraint of $\beta=0.1$. In plot (a), the total data rate is illustrated. Plot (b) illustrates the data rates allocated to individual sensors. The data rates are represented as color-coded blocks.}
    \label{fig:CCP-ADMM01}
\end{figure}

\begin{figure}[t]
\begin{minipage}[b]{0.50\linewidth}
\centering
    \includegraphics[width=\columnwidth]{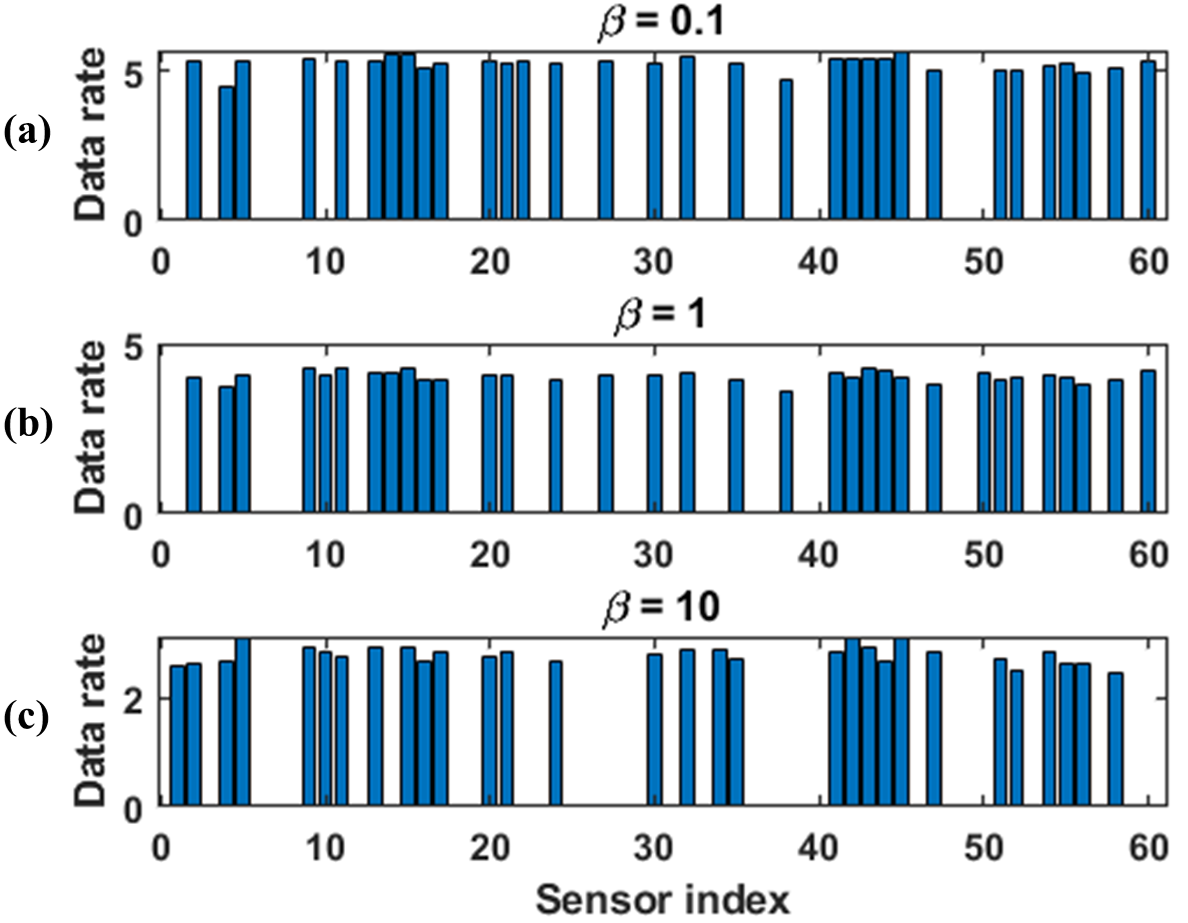}
    \caption{Data rate allocation obtained by CCP with (a) $\beta=0.1$ (b) $\beta=1$ (c) $\beta=10$.}
    \label{fig:final sensor i}
\end{minipage}
\hspace{0.5cm}
\begin{minipage}[b]{0.50\linewidth}
\centering
    \includegraphics[width=\columnwidth]{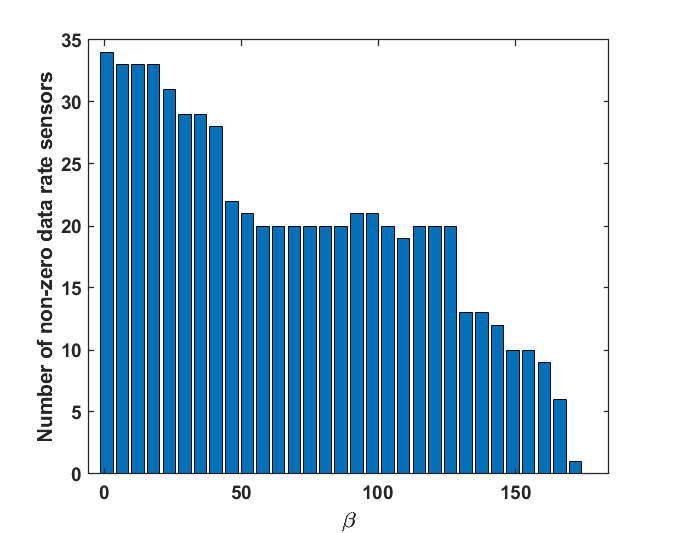}
    \caption{Number of sensors allocated non-zero data rate by CCP tested over the range  $1\leq \beta \leq 180$.}
    \label{fig:nonzero}
\end{minipage}
\end{figure}

Fig.~\ref{fig:final sensor i} shows the data rate allocation obtained at convergence (i.e. after a sufficient number of CCP iterations) for three different values of $\beta$. We observe that a similar subset of sensors are given non-zero data rates for all three $\beta$ values. As one might expect, the allocated data rates to individual sensors decrease as $\beta$ increases. Increasing $\beta$ corresponds to a less stringent demand on the performance of the fusion center's estimator MSE; a less demanding MSE constraint means that satisfactory performance can be achieved with less bitrate. Moreover, it is seen that more sensors are allocated zero data rate as $\beta$ increases.

To confirm this last observation, Fig.~\ref{fig:nonzero} illustrates the number of sensors allocated non-zero data rate after a sufficient number of CCP iterations as $\beta$ is varied from $1$ to $180$. The number of sensors tends to decrease, but is somewhat stagnant in the middle of the beta range. The relationship is not necessarily monotone.
This plot demonstrates a sparsity-promoting property of the proposed data rate allocation method.

\subsection{Target tracking by a drone swarm}
\label{drone-track-tgt}
Next, we consider the problem tracking multiple moving targets using a radar signal and multiple drones. As in \citen{zhan2005}, we consider a scenario in which a base station, which also serves and the data fusion center, illuminates the set of targets with a radar signal. A set of drones, nominally the sensor platforms, make time-delay and Doppler measurements of the radar signal reflected by the targets. The drones transmit these measurements back to the data fusion center, where the targets' positions are estimated by an Extended Kalman Filter (EKF).
Unlike the previous example, the system ``$C$'' matrix in \eqref{eq:measurements} is time-varying as the relative positions of the drones with respect to the targets change over time. Therefore, we consider an approach to recalculate the rate allocation at every time step by an repetitive executions of Algorithm~\ref{alg1}. 
To improve computational efficiency, we adopt a ``warm start'' implementation of Algorithm~\ref{alg1} at every time step, i.e., the optimal allocation from the previous time step is used as the initial condition for the CCP iteration in the next time step.

We assume there are five targets in the entire 2D environment. Each target is assumed to be a point mass driven by a random force which is modeled as an i.i.d. Gaussian process noise. The state vector to be estimated in this simulation study is therefore \eqref{eq:source} with
\[
x_t=\begin{bmatrix} x_t^1 & x_t^2 & x_t^3 & x_t^4 & x_t^5 \end{bmatrix}^\top\in\mathbb{R}^{20}
\]
where $x_t^i=\begin{bmatrix} p_{x,t}^i &  p_{y,t}^i &  v_{x,t}^i &  v_{y,t}^i \end{bmatrix}^\top$ are the position and the velocity of the $i$-th target. Let 
\begin{align*}
A_{Target} = \begin{bmatrix}
1 & 0 & \Delta_{t} & 0\\
0 & 1 & 0 & \Delta_{t}\\
0 & 0 & 1 & 0\\
0 & 0 & 0 & 1
\end{bmatrix}.
\end{align*}
where $\Delta_{t}=1$ is the step size. The $A$ matrix for the five target system is given by $A=\text{diag}(A_{\text{Target}},A_{\text{Target}}, A_{\text{Target}}, A_{\text{Target}}, A_{\text{Target}})$. Likewise, let 
\begin{align*}
F_{Target} = \begin{bmatrix}
\sqrt{10} & 0 & 0 & 0\\
0 & \sqrt{10} & 0 & 0\\
0 & 0 & 1 & 0\\
0 & 0 & 0 & 1
\end{bmatrix}.
\end{align*} The $F$ matrix for the five target system is defined as  $F=\text{diag}(F_{\text{Target}},F_{\text{Target}}, F_{\text{Target}}, F_{\text{Target}}, F_{\text{Target}})$.

\begin{figure}
    \centering
    \includegraphics[width=0.9\columnwidth]{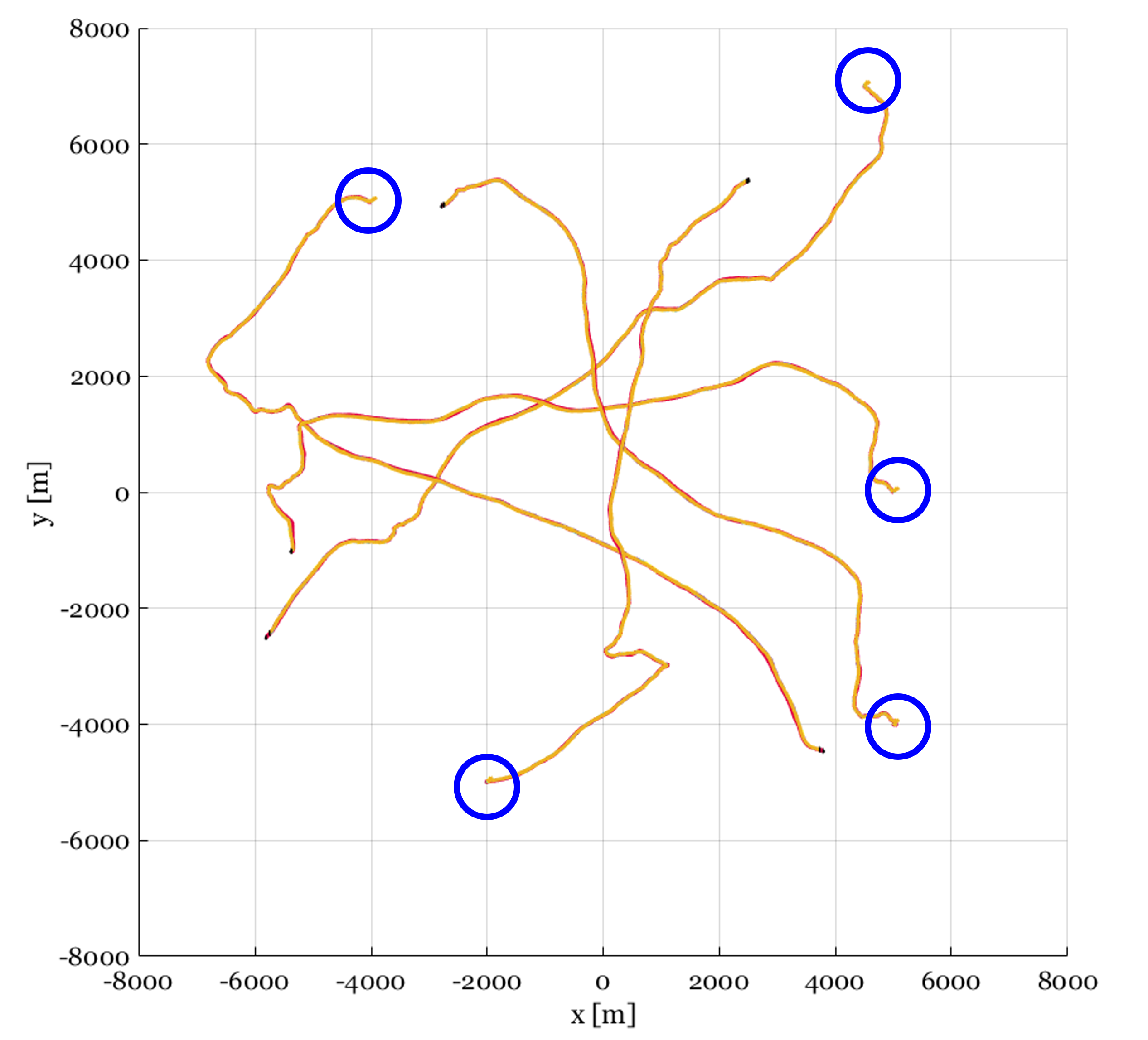}
    \caption{Target tracking simulation result in 5 separate regions. 1 target (red line) and 5 drones (other color lines) are in each region.} 
    \label{fig:5region}
\end{figure}
As shown in Fig.~\ref{fig:5region}, we assume that five drones are available in the neighborhood of each target (so there are 25 drones in total). Each drone is modeled as a point mass in a 2D plane. Let $(p_{x,t}^{i,j}, p_{y,t}^{i,j})$ be the coordinate of the $j$-th drone in the neighborhood of the $i$-th target, and $(v_{x,t}^{i,j}, v_{y,t}^{i,j})$ be its velocity. Drones are controlled by the data fusion center to track the target. In this simulation, we assume that a PD control with acceleration input is used for each drone:
\begin{align}
\label{eq:control}
\begin{split}
a_{x,t}^{i,j} & =K_{P}(\hat{p}_{x,t}^i-p_{x,t}^{i,j})+K_{D}(\hat{v}_{x,t}^i-v_{x,t}^{i,j})-\sum_{k \neq j} L_{P}(p_{x,t}^{i,k}-p_{x,t}^{i,j})-\sum_{k \neq j}L_{D}(v_{x,t}^{i,k}-v_{x,t}^{i,j}) \\ 
a_{y,t}^{i,j} & =K_{P}(\hat{p}_{y,t}^i-p_{y,t}^{i,j})+K_{D}(\hat{v}_{y,t}^i-v_{y,t}^{i,j})-\sum_{k \neq j} L_{P}(p_{y,t}^{i,k}-p_{y,t}^{i,j})-\sum_{k \neq j}L_{D}(v_{y,t}^{i,k}-v_{y,t}^{i,j}).
\end{split}
\end{align}
Here, $(\hat{p}_{x,t}^i,\hat{p}_{y,t}^i,\hat{v}_{x,t}^i,\hat{v}_{y,t}^i)$ is the estimated state of the target $i$ computed by the data fusion center using the EKF described below. The gains $K_P$ and $K_D$ are tuned to keep the drones close to the target, whereas gains $L_P$ and $L_D$ are selected so that drones in the same neighborhood stay away from each other.

At every time step, the base station illuminates the targets with a radar signal. We assume that the reflected signal from the $i$-th target is observed by the drones in the $i$-th region. Each drone in the $i$-th region obtains a two-dimensional measurement of the $i$-th target. The $j$-th drone in the $i$-th region receives Doppler $f_t^{i,j}$ and time delay $\tau_t^{i,j}$ measurements of the reflected radar signal, which are given by \citen{zhan2005}:
\begin{align}
f_t^{i,j}= & \frac{(v_{x,t}^i p_{x,t}^{i}+v_{y,t}^i p_{y,t}^{i})}{\sqrt{(p_{x,t}^{i})^2+(p_{y,t}^{i})^2}} +\frac{(v_{x,t}^i-v_{x,t}^{i,j})(p_{x,t}^i-p_{x,t}^{i,j})+(v_{y,t}^i-v_{y,t}^{i,j})(p_{y,t}^i-p_{y,t}^{i,j})}{\sqrt{(p_{x,t}^i-p_{x,t}^{i,j})^2+(p_{y,t}^i-p_{y,t}^{i,j})^2}}  \label{eq:doppler} \\
\tau_t^{i,j}= & \sqrt{(p_{x,t}^i)^2+(p_{y,t}^i)^2} + \sqrt{(p_{x,t}^i-p_{x,t}^{i,j})^2+(p_{y,t}^i-p_{y,t}^{i,j})^2}. \label{eq:delay}
\end{align}
Since \eqref{eq:doppler} and \eqref{eq:delay} are nonlinear observations of the state vector $x_t$, we linearize them around the current best estimate $(\hat{p}_{x,t}^i,\hat{p}_{y,t}^i,\hat{v}_{x,t}^i,\hat{v}_{y,t}^i)$, where these estimates are recursively computed using the EKF.

\begin{figure}
    \centering
    \includegraphics[width=0.9\columnwidth]{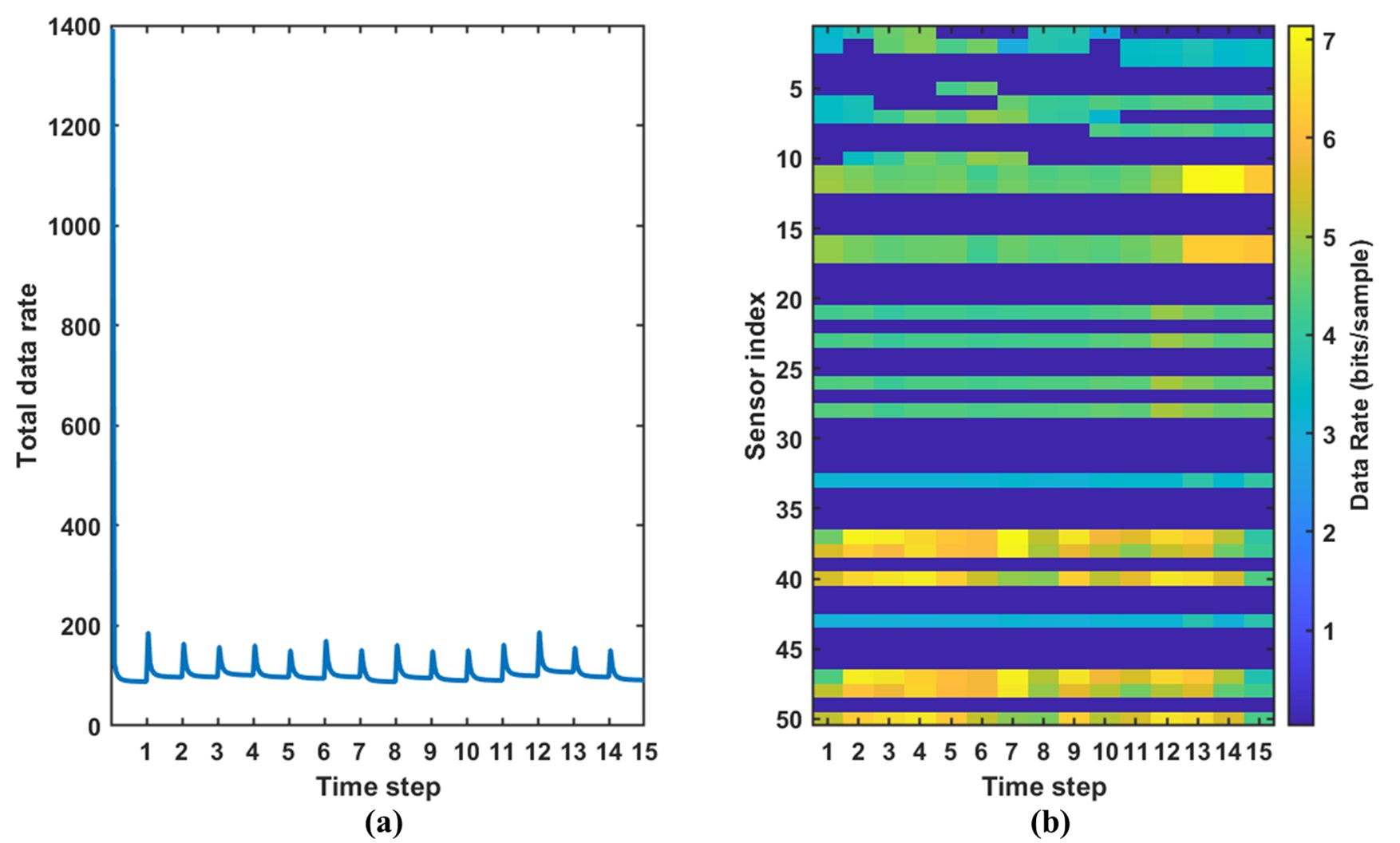}
    \caption{Optimized data rates over 15 time steps, with 15 iterations of the CCP algorithm performed per timestep. The MSE constraint is $\beta=1$. Fig.~\ref{fig:CCP-ADMM-tgt1}~(a) demonstrates the convergence of the CCP algorithm from the warm start solutions---when the time index on the horizontal axis updates from time $t-1$ to time $t$, the optimization is re-started for the system matrices at time $t$. In between times $t$ and $t+1$, we plot the communication cost obtained for fifteen iterations of the CCP algorithm. Fig.~\ref{fig:CCP-ADMM-tgt1}~(b) illustrates the data rates assigned to each sensor at the convergence of the CCP algorithm (nominally 15 iterations) at each time step. The assigned data rates are represented by color-coded blocks.} 
    \label{fig:CCP-ADMM-tgt1}
\end{figure}
Fig.~\ref{fig:5region} shows the scenario described above simulated for 1,000 time steps. Blue circles represent neighborhood regions of each target, and there are five drones in each region.
The target trajectory is presented as red line and drone trajectories in other colors.
Based on the calculated position and velocity of target and drones, we apply the CCP iterations in each time step. Fig.~\ref{fig:CCP-ADMM-tgt1} shows the rate allocation obtained as the CCP iterations converge in each time step. In
Fig.~\ref{fig:CCP-ADMM-tgt1}~(a) we plot the total data rate allocation, while in Fig.~\ref{fig:CCP-ADMM-tgt1}~(b) we illustrate the assigned data rates to individual sensors, at each time step, at the convergence of the CCP algorithm (nominally 15 iterations).  
Note that there are 50 possible sensors, since each of the 25 drones receives both range and Doppler measurements. The result changes over time, reflecting the time-varying nature of the considered problem.

To gain further insight, Fig~\ref{fig:zoomed} shows the data rate allocated to drones in region 1 and their spatial positions at time step $t=70$. Fig~\ref{fig:zoomed}~(a) indicates that non-zero data rates (yellow color block) are given only to sensors 2,3,7 and 8, which correspond to the delay and Doppler measurements obtained from drones 2 and 3. Fig~\ref{fig:zoomed}~(c) indicates that drones 2 and 3 are the left-most and right-most vehicles in the formation.  
\begin{figure}
    \centering
    \includegraphics[width=0.9\columnwidth]{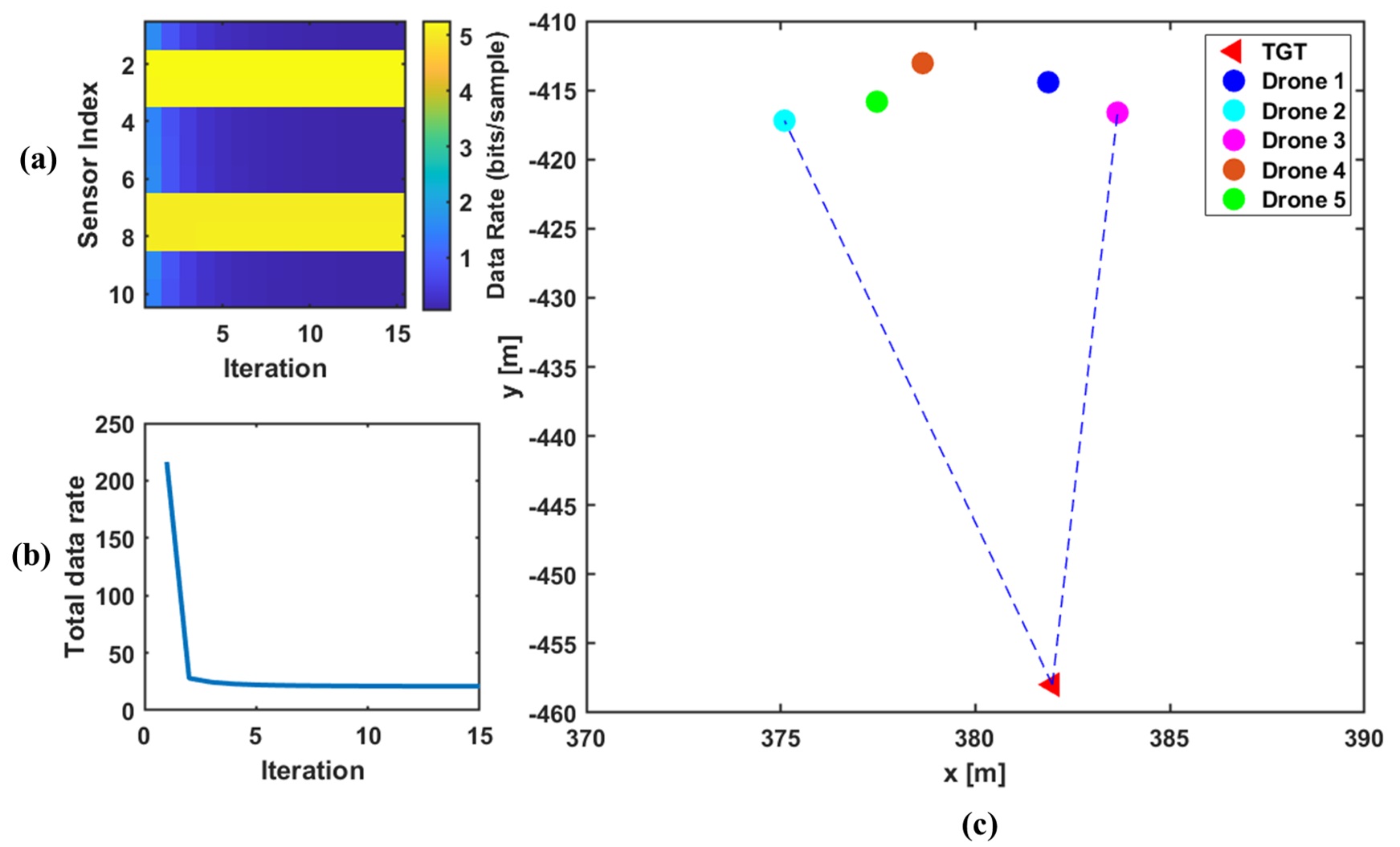}
    \caption{ Data rate allocation in Region 1 at t=70 ($\beta=1$). In (a), we show the data rate allocation to each sensor after a given number of CCP iterations. Indices 1-5 in (a) correspond to the time delay measurements $\tau_t^{1,j}, j=1, ..., 5$, while indices 6-10 correspond to the Doppler estimates $f_t^{1,j}, j=1, ..., 5$. In (b), we plot the total data rate allocation against the number of CCP iterations. In (c) we plot the physical position of the target (TGT) and drones. Dotted lines are drawn from drones allocated a nonzero data rate to the target. It is seen that the two drones allocated nonzero data rates have the largest possible angular spread, and, in a sense, the most ``different" measurements, among pairs of drones. } 
    \label{fig:zoomed}
\end{figure}

\section{Discussion}\label{sec:discussion}
As observed in Section~\ref{sec:Num}, the SRA formulation (\ref{eq:s3ra_inf}) (similarly,  (\ref{eq:s3ra})) admits sparse solutions, i.e., for several sensors we have $\delta_i=0$. In this section, we take a closer look at the mathematical structure of the SRA problem. We first remark that the sparsity-promoting 
nature of the solution can be attributed to the mathematical structure of the communication cost. The communication cost in (\ref{eq:new_prob1_a}) is reminiscent of the \textit{sum-of-logs} regularizer, a widely used heuristic to induce sparse solutions \citen{candes2008enhancing}.  




To develop further insight, we consider a special case of sensor rate allocation problem for which the closed-form solution is available. Consider a scalar system described by 
\begin{align*}
    \bx_{t+1}=& a \bx_t +f \bw_t, \quad \bw_t \stackrel{i.i.d.}{\sim} \mathcal{N}(0,1),
\end{align*}
and a single measurement   
\begin{align*}
\boldeta_t =& \bx_t + \bv_t, \quad \bv_t \stackrel{i.i.d.}{\sim} \mathcal{N}(0,V_t).
\end{align*}
Since the system is observable, $P\triangleq \lim_{t \rightarrow \infty} P_{t|t}$  exists and is computed by the algebraic Riccati equation (ARE)
\begin{equation}
\label{eq:ARE}
    P^{-1}= (a^2P+f^2)^{-1}+V^{-1},
\end{equation}
where $V \triangleq \limsup_{t\rightarrow \infty} V_t$. The time-invariant rate allocation problem (\ref{eq:s3ra_inf}) for this  system is 
\begin{equation}
 \label{eq:SRA_scalar}
 \begin{split}
     \min_{P,V \geq 0} \quad & \log (a^2+\frac{f^2}{P})\\
     \text{s.t.} \quad & P \leq \beta\\
     &P^{-1}= (a^2P+f^2)^{-1}+V^{-1}.
 \end{split}
 \end{equation}
 Denote by $P=g(V)$ the unique positive solution to the ARE (\ref{eq:ARE}). It can be shown that $g(V)$ is a strictly increasing, and thus invertible, function of $V$ with a bounded range $[g(0)=0,g(\infty)= \frac{f^2}{1-a^2}]$. Therefore, problem (\ref{eq:SRA_scalar}) can be equivalently written as 
\begin{equation}
\label{eq:equi_prob}
 \begin{split}
     \min_{P\geq 0} \quad &  \log (a^2+\frac{f^2}{P})\\
     \text{s.t.} \quad &  P \leq \beta \\
      &P \leq \frac{f^2}{1-a^2}.  
 \end{split}
 \end{equation}
 It is straightforward to verify that the minimizer of (\ref{eq:equi_prob}) is $P^*= \min \{\beta, \frac{f^2}{1-a^2}\}$
 and the optimal value of (\ref{eq:equi_prob}) is $\min\{0,\; \log(a^2+\frac{f^2}{\beta})\}$. The $P^*$ is obtained by adopting $V^*=g^{-1}(P^*)$. The condition $P^*=\frac{f^2}{1-a^2}$, leading to $V^*=\infty$, corresponds to the condition when it is optimal to allot zero data rate to the sensor. This partly explains the sparse nature of the solutions we observed in Section~\ref{sec:Num}. It is worth noting that problem (\ref{eq:equi_prob}) is equivalent to the scalar infinite-horizon sequential rate-distortion problem studied in \citen{tanaka2016semidefinite,tatikonda2004stochastic}. 

\section{Conclusion}
\label{sec:conclusion}
In this paper, we considered a sensor data rate allocation problem for dynamic sensor fusion over resource restricted communication networks. 
Using a system model motivated by practical remote sensing scenarios and well-understood methods of zero delay source coding, we proposed a data rate allocation problem between a group of remote sensors and a fusion center. We reformulated this problem using information-theoretic tools and ideas from Kalman filtering. Our proposed optimization took the form of a difference-of-convex program. We applied the CCP algorithm to find a heuristic solution. We conducted a numerical study on a 1D heat transfer model and on a 2D target tracking by drone swarm scenario. Our numerical results suggest that the proposed approach is sparsity-promoting. By considering a a limiting case of the SRA problem that admitted an analytical solution, we gained insight into the sparsity-promoting property.

Possible directions for future research include analyzing the non-convexity of the SRA problem and developing computationally efficient algorithms to a perform the minimization. In particular, when considering the infinite horizon sensor resource allocation problem, we restricted our optimization to time-invariant rate allocations (i.e. time-invariant quantizer sensitivities). The existence (or lack thereof) of time-varying (e.g., periodic) solutions that outperform optimal time-invariant solutions is an interesting question that presents an opportunity for future work. In \citen{tanakaSingleLetter}, the optimality of time-invariant solutions is established for a related (but not equivalent) problem called the sequential rate-distortion (SRD) problem. The convexity of the multi-stage SRD problem allows one to effectively "single-letterize" the the optimization. Unfortunately, in our current setting, the multi-stage SRA problem (e.g., equation (\ref{eq:prob1})) is not a convex program. This makes it difficult to apply the proof techniques from \citen{tanakaSingleLetter} to the present setup. 

Extensions to unreliable networks are attractive avenues for future research. Notably, our notion of airtime in Section \ref{ssec:phy}  relied infinite blocklength expressions for channel capacity. It would be interesting to incorporate modern tools from finite blocklength information theory (cf. \citen{polyanskiy}) to better quantify ``real world" communication costs and notions of reliability. Likewise, our numerical simulations in Section \ref{sec:Num} assumed reliable communication channels. Incorporating realistic physical layer channel impairments into simulations like those in Section \ref{sec:Num} presents another exciting opportunity for future investigations. Finally, it may also be possible to discover a more illuminating connection to the CEO problem. 

\section*{Acknowledgements}
This work was supported by the Defense Advanced Research Projects Agency (DARPA) under Grant D19AP00078 and National Science Foundation (NSF) under Award 1944318. 

\section*{CONFLICT OF INTEREST}
The authors declare no potential conflict of interests.

\section*{Data Availability Statement}
The data that support the findings of this study are available from the corresponding author upon reasonable request.

\bibliography{wileyNJD_AMA}



\appendix
\section{.\ Proof of Lemma 1\label{app1}}

\label{proof:lem:rate}
Recall that $\boldeta_{i,t}=\mathbf{q}_{i,t}-\boldxi_{i,t}$. Thus 
\begin{align}
    H(\mathbf{q}_{i,t}|\boldxi_{i,t})=H(\boldeta_{i,t}|\boldxi_{i,t})
\end{align} Thus, by (\ref{eq:entropy}), we have
\begin{align}\label{eq:newentropy}
    H(\boldeta_{i,t}|\boldxi_{i,t})\leq \mathbb{E}(\boldell_{i,t}) < H(\boldeta_{i,t}|\boldxi_{i,t})+1.
\end{align} Under the joint distribution of $\boldtheta_{i,t}$ and $\boldeta_{i,t}$ induced by the model in Fig.~\ref{fig:channel1} it can be shown that (cf.  \citen{zamir1992},  
[\citem{tanaka2016}, Lemma 1]) 
\begin{align}\label{eq:entropyMIeq}
    H(\boldeta_{i,t}|\boldxi_{i,t})=I(\boldtheta_{i,t};\boldeta_{i,t}).
\end{align} The results of the previous section show that the joint distribution of $\boldtheta_{i,t}$ and $\boldeta_{i,t}$ induced by the model in Fig.~\ref{fig:channel2} is equivalent to the joint distribution induced by Fig.~\ref{fig:channel1}. Substituting (\ref{eq:entropyMIeq}) into (\ref{eq:newentropy}) completes the proof. 

\section{. \ Proof of Lemma 2\label{app2}}

\label{proof:lem:silva}
To see the upper bound, first expand the mutual information in terms of differential entropy
\begin{align}\label{eq:mideng}
I(\boldtheta_{i,t};\boldeta_{i,t}) &= h(\boldeta_{i,t})-h(\boldeta_{i,t}|\boldtheta_{i,t}) \\
&=h(\boldeta_{i,t})-h(\boldeta_{i,t}-\boldtheta_{i,t}|\boldtheta_{i,t})\label{eq:dedef} \\
&=h(\boldeta_{i,t})-h(\bv_{i,t})\label{eq:deind},
\end{align}where (\ref{eq:dedef}) follows from the definition of $\boldeta_{i,t}$ and (\ref{eq:deind}) follows since $\bv_{i,t}$ and $\boldtheta_{i,t}$ are independent (cf. below (\ref{eq:last22})).
Likewise 
\begin{align}\label{eq:mideg}
    I(\boldtheta^\text{G}_{i,t};\boldeta^{\text{G}}_{i,t})= h(\boldeta_{i,t}^\text{G})-h(\bv_{i,t}^\text{G}).
\end{align}

Gaussian random variables have the maximum entropy among all random variables with the same variance (cf. [\citem{cover2012elements}, Theorem 8.6.5]). Thus, we have
\begin{align}\label{eq:gei}
    h(\boldeta_{i,t})\le h(\boldeta_{i,t}^\text{G}).
\end{align}
Furthermore, by applying the definition of differential entropy, expanding the PDF of $\bv_{i,t}^\text{G}$, and recalling that $\bv_{i,t}^\text{G}$ and $\bv_{i,t}$ have the same variance, we have 
\begin{align}\label{eq:kltrick}
    h(\bv_{i,t})&=h(\bv_{i,t}^\text{G})-D(\bv_{i,t}\|\bv_{i,t}^\text{G}) \\
    &=h(\bv_{i,t}^\text{G})-\frac{1}{2}\log\frac{2\pi e}{12}.
\end{align}
Thus, recalling both (\ref{eq:mideg}) and (\ref{eq:mideng}), and subtracting both sides of (\ref{eq:kltrick}) from the respective sides of (\ref{eq:gei}) gives
\begin{align*}
    I(\boldtheta_{i,t};\boldeta_{i,t})&\leq h(\boldeta_{i,t}^\text{G})-h(\bv_{i,t}^\text{G})+\frac{1}{2}\log\frac{2\pi e}{12} \nonumber \\
    &=I(\boldtheta_{i,t}^\text{G};\boldeta_{i,t}^\text{G})+\frac{1}{2}\log\frac{2\pi e}{12}. 
\end{align*}
On the other hand, the lower bound follows from [\citem{silva2009}, Lemma C.1].

\section{. \ Proof of Proposition 1.}\label{app4}
\label{proof:lem:ineq}
As the first step of the proof, we show by induction that for $t=1, ... , T$, we have  
\begin{subequations}
\label{eq:claim}
\vspace{-1ex}
\begin{align}
    &Q_{t|t}^{**}\succeq Q_{t|t}^{*}. \label{eq:claim_a} \\
    &Q_{t+1|t}^{**}\succeq Q_{t+1|t}^{*}. \label{eq:claim_b}
\end{align}
\end{subequations}
For the initial step, \eqref{eq:claim_a} trivially holds  as $Q_{1|1}^{**} = Q_{1|1}^*$  by construction. 
Note that $Q_{2|1}^{**-1} = AQ_{1|1}^{**-1}A^\top+FF^\top$ by construction, and $Q_{2|1}^{*-1}\succeq AQ_{1|1}^{*-1}A^\top+FF^\top$ 
since $Q_{1|1}^*$ and $Q_{2|1}^*$ are feasible solution to \eqref{eq:prob1_f}. Thus $Q_{2|1}^{**-1} \preceq Q_{2|1}^{*-1}$ or equivalently $Q_{2|1}^{**} \succeq Q_{2|1}^{*}$.
Therefore, relation \eqref{eq:claim} holds for $t=1$. For the induction step, we assume \eqref{eq:claim} holds for $t=k(\geq 1)$:
\begin{subequations}
\label{eq:claim_k}
\begin{align}
    &Q_{k|k}^{**} \succeq Q_{k|k}^{*}, \label{eq:claim_k1}\\
    &Q_{k+1|k}^{**} \succeq Q_{k+1|k}^{*}, \label{eq:claim_k2}
\end{align}
\end{subequations}
and we show that \eqref{eq:claim} holds for $t=k+1$. From 
\eqref{eq:prob1_e} and \eqref{eq:new_prob1_e}, we have
\begin{align*}
Q_{k+1|k+1}^{**}&=Q_{k+1|k}^{**}+\sum\nolimits_{i=1}^M \delta_{i,k+1}^{*}C_i^\top C_i \text{ and } \\
Q_{k+1|k+1}^{*}&=Q_{k+1|k}^{*}+\sum\nolimits_{i=1}^M \delta_{i,k+1}^{*}C_i^\top C_i
\end{align*}
Thus, \eqref{eq:claim_k2} directly implies
\begin{equation}
\label{eq:claim1a}
    Q_{k+1|k+1}^{**} \succeq Q_{k+1|k+1}^{*}.
\end{equation}
Substituting $Q_{k+1|k+1}^{**}$ into 
\eqref{eq:prob1_e}, we have
\begin{subequations}
\label{eq:claim_chain}
\begin{align}
Q_{k+2|k+1}^{**-1}&=AQ_{k+1|k+1}^{**-1}A^\top + FF^\top \label{eq:claim_chain1}\\
&\preceq AQ_{k+1|k+1}^{*-1}A^\top + FF^\top  \label{eq:claim_chain2} \\
&\preceq Q_{k+2|k+1}^{*-1}, \label{eq:claim_chain3}
\end{align}
\end{subequations}
where \eqref{eq:claim_chain2} is due to \eqref{eq:claim1a}. The last inequality holds since 
$(Q_{k+1|k+1}^*, Q_{k+2|k+1}^*)$ is a feasible solution satisfying \eqref{eq:new_prob1_e}. \eqref{eq:claim_chain} implies
\begin{equation}
\label{eq:claim1b}
Q_{k+2|k+1}^{**}\succeq Q_{k+2|k+1}^{*}.
\end{equation}
Inequalities \eqref{eq:claim1a} and \eqref{eq:claim1b} establish the relation \eqref{eq:claim} for $t=k+1$ and thus for any $t=1: T$. This completes the first step of the proof.

The proof can be completed as follows. Let $J_2^*$ be the optimal value of \eqref{eq:new_prob1} attained by  $[\psi_{i,t}^{*}, \delta_{i,t}^{*}, \gamma_{i,t}^{*}, S_t^{*}, Q_{t|t}^{*}, Q_{t|t-1}^{*}]$. From \eqref{eq:claim}, we have
\begin{align*}
&0\preceq
\begin{bmatrix}
    \delta_{i,t}^{*}-\gamma_{i,t}^{*} & \delta_{i,t}^{*}C_i \\
C_i^\top \delta_{i,t}^{*} & Q_{t|t-1}^{*}+C_i^\top \delta_{i,t}^{*} C_i
\end{bmatrix} \preceq
\begin{bmatrix}
     \delta_{i,t}^{*}-\gamma_{i,t}^{*} & \delta_{i,t}^{*}C_i \\
C_i^\top \delta_{i,t}^{*} & Q_{t|t-1}^{**}+C_i^\top \delta_{i,t}^{*} C_i
\end{bmatrix}
\end{align*}
and
\begin{align*}
 0\preceq
\begin{bmatrix}
 S_t^{*} & I \\ I & Q_{t|t}^{*}
\end{bmatrix} \preceq \begin{bmatrix}
 S_t^{*} & I \\ I & Q_{t|t}^{**}
\end{bmatrix}.   
\end{align*}
Moreover, $(Q_{t|t}^{**},Q_{t|t-1}^{**})$ satisfies the equality constraints \eqref{eq:definition} by construction.
This means that $[\psi_{i,t}^*, \delta_{i,t}^{*}, \gamma_{i,t}^{*}, S_t^{*}, Q_{t|t}^{**}, Q_{t|t-1}^{**}]$ is a feasible solution of equation \eqref{eq:prob1} and it obtains the value $J_2^*$. Therefore, $J_1^* \leq J_2^*$. However, \eqref{eq:prob1} has more strict constraint than \eqref{eq:new_prob1} which immediately implies $J_1^* \geq J_2^*$. Therefore, $J_1^* = J_2^*$.

\end{document}